\documentclass[11pt,reqno]{amsart}
\usepackage{amsmath,amsfonts,amsthm,amssymb}
\usepackage{graphicx}
\usepackage{verbatim}
\usepackage{color}
\usepackage{amscd}
\usepackage{verbatim}

\begin{document}
\newcommand{\M}{{\mathcal M}}
\newcommand{\loc}{{\mathrm{loc}}}
\newcommand{\core}{C_0^{\infty}(\Omega)}
\newcommand{\sob}{W^{1,p}(\Omega)}
\newcommand{\sobloc}{W^{1,p}_{\mathrm{loc}}(\Omega)}
\newcommand{\merhav}{{\mathcal D}^{1,p}}
\newcommand{\be}{\begin{equation}}
\newcommand{\ee}{\end{equation}}
\newcommand{\mysection}[1]{\section{#1}\setcounter{equation}{0}}
%%%%%%%%%%%%
\newcommand{\laplace}{\Delta}%{\triangle}
\newcommand{\pl}{\laplace_p}
\newcommand{\grad}{\nabla}%{\bigtriangledown}
\newcommand{\pd}{\partial}
\newcommand{\bo}{\pd}
\newcommand{\csub}{\subset \subset}
\newcommand{\sm}{\setminus}
\newcommand{\ssm}{:}
\newcommand{\diver}{\mathrm{div}\,}
%%%%%%%%%%%%%%%
\newcommand{\bea}{\begin{eqnarray}}
\newcommand{\eea}{\end{eqnarray}}
\newcommand{\bean}{\begin{eqnarray*}}
\newcommand{\eean}{\end{eqnarray*}}
\newcommand{\thkl}{\rule[-.5mm]{.3mm}{3mm}}
%%%%%%%%%%%%%%%%%%%%%%%%%%%
\newcommand{\cw}{\stackrel{\rightharpoonup}{\rightharpoonup}}
\newcommand{\id}{\operatorname{id}}
\newcommand{\supp}{\operatorname{supp}}
\newcommand{\wlim}{\mbox{ w-lim }}
\newcommand{\mymu}{{x_N^{-p_*}}}
\newcommand{\R}{{\mathbb R}}
\newcommand{\N}{{\mathbb N}}
\newcommand{\Z}{{\mathbb Z}}
\newcommand{\Q}{{\mathbb Q}}
\newcommand{\abs}[1]{\lvert#1\rvert}
%%%%%%%%%%%
\newtheorem{theorem}{Theorem}[section]
\newtheorem{corollary}[theorem]{Corollary}
\newtheorem{lemma}[theorem]{Lemma}
\newtheorem{notation}[theorem]{Notation}
\newtheorem{definition}[theorem]{Definition}
\newtheorem{remark}[theorem]{Remark}
\newtheorem{proposition}[theorem]{Proposition}
\newtheorem{assertion}[theorem]{Assertion}
\newtheorem{problem}[theorem]{Problem}
%%%%%%%%%%%%%%%%%%
\newtheorem{conjecture}[theorem]{Conjecture}
\newtheorem{question}[theorem]{Question}
\newtheorem{example}[theorem]{Example}
\newtheorem{Thm}[theorem]{Theorem}
\newtheorem{Lem}[theorem]{Lemma}
\newtheorem{Pro}[theorem]{Proposition}
\newtheorem{Def}[theorem]{Definition}
\newtheorem{Exa}[theorem]{Example}
\newtheorem{Exs}[theorem]{Examples}
\newtheorem{Rems}[theorem]{Remarks}
\newtheorem{Rem}[theorem]{Remark}

\newtheorem{Cor}[theorem]{Corollary}
\newtheorem{Conj}[theorem]{Conjecture}
\newtheorem{Prob}[theorem]{Problem}
\newtheorem{Ques}[theorem]{Question}
\newtheorem*{corollary*}{Corollary}
\newtheorem*{theorem*}{Theorem}
\newcommand{\pf}{\noindent \mbox{{\bf Proof}: }}

%%%%%%%%%%%%%%%%%%
%\newenvironment{proof}{{\bf Proof.}}{\hfill $\bowtie$\vskip4mm}

\renewcommand{\theequation}{\thesection.\arabic{equation}}
\catcode`@=11 \@addtoreset{equation}{section} \catcode`@=12
%%%%%%%%%%%%%%%%%%%%%%
\newcommand{\Real}{\mathbb{R}}
\newcommand{\real}{\mathbb{R}}
\newcommand{\Nat}{\mathbb{N}}
\newcommand{\ZZ}{\mathbb{Z}}
\newcommand{\CC}{\mathbb{C}}
\newcommand{\Pess}{\opname{Pess}}
\newcommand{\Proof}{\mbox{\noindent {\bf Proof} \hspace{2mm}}}
\newcommand{\mbinom}[2]{\left (\!\!{\renewcommand{\arraystretch}{0.5}
\mbox{$\begin{array}[c]{c}  #1\\ #2  \end{array}$}}\!\! \right )}
\newcommand{\brang}[1]{\langle #1 \rangle}
\newcommand{\vstrut}[1]{\rule{0mm}{#1mm}}
\newcommand{\rec}[1]{\frac{1}{#1}}
\newcommand{\set}[1]{\{#1\}}
\newcommand{\dist}[2]{$\mbox{\rm dist}\,(#1,#2)$}
\newcommand{\opname}[1]{\mbox{\rm #1}\,}
\newcommand{\mb}[1]{\;\mbox{ #1 }\;}
\newcommand{\undersym}[2]
 {{\renewcommand{\arraystretch}{0.5}  \mbox{$\begin{array}[t]{c}
 #1\\ #2  \end{array}$}}}
\newlength{\wex}  \newlength{\hex}
\newcommand{\understack}[3]{%
 \settowidth{\wex}{\mbox{$#3$}} \settoheight{\hex}{\mbox{$#1$}}
 \hspace{\wex}  \raisebox{-1.2\hex}{\makebox[-\wex][c]{$#2$}}
 \makebox[\wex][c]{$#1$}   }%
%%Macros for changing font size in math.
\newcommand{\smit}[1]{\mbox{\small \it #1}}% only for letters, numbers
\newcommand{\lgit}[1]{\mbox{\large \it #1}}% only for letters, numbers
\newcommand{\scts}[1]{\scriptstyle #1}
\newcommand{\scss}[1]{\scriptscriptstyle #1}
\newcommand{\txts}[1]{\textstyle #1}
\newcommand{\dsps}[1]{\displaystyle #1}
%%%%%%%%%%%%%%%%%%%%%%%%%%%%%%%%%%%%%%%%%%%%%%
\newcommand{\dx}{\,dx}
\newcommand{\dy}{\,dy}
\newcommand{\dz}{\,\mathrm{d}z}
\newcommand{\dt}{\,dt}
\newcommand{\dr}{\,\mathrm{d}r}
\newcommand{\du}{\,du}
\newcommand{\dv}{\,\mathrm{d}v}
\newcommand{\dV}{\,\mathrm{d}V}
\newcommand{\ds}{\,\mathrm{d}s}
\newcommand{\dS}{\,\mathrm{d}S}
\newcommand{\dk}{\,\mathrm{d}k}

\newcommand{\dphi}{\,\mathrm{d}\phi}
\newcommand{\dtau}{\,\mathrm{d}\tau}
\newcommand{\dxi}{\,\mathrm{d}\xi}
\newcommand{\deta}{\,\mathrm{d}\eta}
\newcommand{\dsigma}{\,\mathrm{d}\sigma}
\newcommand{\dtheta}{\,\mathrm{d}\theta}
\newcommand{\dnu}{\,\mathrm{d}\nu'}

%%%%%%%%%%%%%%%%%%%%%%%%%%%%%%%Macros for Greek letters.
\def\ga{\alpha}     \def\gb{\beta}       \def\gg{\gamma}
\def\gc{\chi}       \def\gd{\delta}      \def\ge{\epsilon}
\def\gth{\theta}                         \def\vge{\varepsilon}
\def\gf{\phi}       \def\vgf{\varphi}    \def\gh{\eta}
\def\gi{\iota}      \def\gk{\nu}      \def\gl{\lambda}
\def\gm{\mu}        \def\gn{\nu'}         \def\gp{\pi}
\def\vgp{\varpi}    \def\gr{\rho}        \def\vgr{\varrho}
\def\gs{\sigma}     \def\vgs{\varsigma}  \def\gt{\tau}
\def\gu{\upsilon}   \def\gv{\vartheta}   \def\gw{\omega}
\def\gx{\xi}        \def\gy{\psi}        \def\gz{\zeta}
\def\Gg{\Gamma}     \def\Gd{\Delta}      \def\Gf{\Phi}
\def\Gth{\Theta}
\def\Gl{\Lambda}    \def\Gs{\Sigma}      \def\Gp{\Pi}
\def\Gw{\Omega}     \def\Gx{\Xi}         \def\Gy{\Psi}
%%%%%%%%%%%%%%%%%%%%%%%%%%%%%%%%%%%%%%%%%%%%%%%%%%%%%%%%%%%%%%%%%%

\renewcommand{\div}{\mathrm{div}}
\newcommand{\red}[1]{{\color{red} #1}}

\newcommand{\cqfd}{\begin{flushright}                  
			 $\Box$
                 \end{flushright}}

%\begin{titlepage}

\title{Heat kernel and Riesz transform of Schr\"{o}dinger operators}

\author{Baptiste Devyver}
\address{Baptiste Devyver, Department of Mathematics, University of British Columbia, Vancouver, B.C.
Canada V6T 1Z2}
\email{devyver@math.ubc.ca ; devyver@tx.technion.ac.il}

\begin{abstract}

The goal of this article is two-fold: in a first part, we prove Gaussian estimates for the heat kernel of Schr\"{o}dinger operators $\Delta+\mathcal{V}$ whose potential $\mathcal{V}$ is ``small at infinity'' in an integral sense. In a second part, we prove sharp boundedness result for the associated Riesz transform with potential $d(\Delta+\mathcal{V})^{-1/2}$. A characterization of $p$-hyperbolicity, which is of independent interest, is also proved.



\end{abstract}

\maketitle

\section{Introduction}

In this article, we shall be concerned with heat kernel estimates and Riesz transform for Schr\"{o}dinger operators. Starting from the seminal work of T. Coulhon and X.T. Duong \cite{CDu}, numerous results in the literature show that these two topics are tightly intertwined. Loosely speaking, Gaussian estimates for the heat kernel imply boundedness of the Riesz transform on $L^p$ for $p\in (1,2)$, while Gaussian estimates for the heat kernel of the Hodge Laplacian $\vec{\Delta}=dd^*+d^*d$, acting on $1$-forms imply boundedness of the Riesz transform on $L^p$ for $p\in (2,\infty)$ (see \cite{CDu}, \cite{CDu2}). Also, if $\mathcal{V}(x)$ is the minimum of $0$ and the smallest negative eigenvalue of the (symmetric) Ricci curvature operator at $x\in M$, then Gaussian estimates for the Schr\"{o}dinger operator $\Delta+\mathcal{V}$ imply that the heat kernel of $\vec{\Delta}$ has Gaussian estimates, and thus the Riesz transform is bounded on $L^p$ for $p\in (2,\infty)$ (see \cite{CZ}, \cite{D1}).

Gaussian estimates for the heat kernel of a Schr\"{o}dinger operator have been investigated for quite a long time (see e.g. \cite{S}, \cite{DS}, \cite{CZ}, \cite{T1}, \cite{Z1}, \cite{Z2}), both by probabilistic and analytic methods. In \cite{CZ}, it is shown that if $\mathcal{V}$ is negative and small in an integral sense, then the heat kernel of $\Delta+\mathcal{V}$ has Gaussian estimates. A sharper result is proved by Takeda \cite{T2} by probabilistic methods: the smallness of $\mathcal{V}$ {\em at infinity} (in an integral sense) is enough to have the Gaussian estimates for the heat kernel of $\Delta+\mathcal{V}$. We will improve Takeda's result (Theorem \ref{heat_kernel}), and as a consequence we will prove a boundedness result for the Riesz transform whose Ricci curvature is ``small at infinity'' in an integral sense (Corollary \ref{Riesz_small}). Our approach is purely analytic, and owes much to the theory of perturbation of Schr\"{o}dinger operators developed by Y. Pinchover and M. Murata (in particular, \cite{P1}, \cite{P2}, \cite{P5}, \cite{P3} and \cite{M}). 

In the second part of the article, we will be concerned with the boundedness on $L^p$ of a Riesz transform with potential $d(\Delta+\mathcal{V})^{1/2}$. When the potential $\mathcal{V}$ is non-negative and lies in a reverse H\"{o}lder class, the operator $d(\Delta+\mathcal{V})^{1/2}$ have been studied by Shen \cite{Sh}, Auscher and Ben Ali \cite{AB}, and Badr and Ben Ali \cite{BBA}. If $\mathcal{V}$ can take negative values, it is customary to assume that it is ``subcritical'' in some sense, and under this kind of assumptions, boundedness results have been obtained by Assaad \cite{A}, and Assaad and Ouhabaz \cite{A}. Their results are not sharp however, in the sense that they do not prove whether or not the obtained range of boundedness for $d(\Delta+\mathcal{V})^{-1/2}$ is the largest possible. Our aim is to improve their results in order to obtain a {\em sharp} result (Theorem \ref{main4}), for potentials $\mathcal{V}$ taking possibly negative values that, roughly speaking, are ``small at infinity'' in an integral sense. We are motivated by the results of Guillarmou and Hassell \cite{GuH}, who prove a sharp boundedness result for $d(\Delta+\mathcal{V})^{-1/2}$ on an asymptotically conic manifold, for potentials $\mathcal{V}$ decaying at rate $O(|x|^{-3})$ at infinity. Precisely, they prove that (in absence of zero-modes and zero-resonances and in dimension larger that $3$), the operator $d(\Delta+\mathcal{V})^{-1/2}$ is bounded on $L^p$ if and only if $p\in (1,n)$, where $n$ is the dimension of the manifold. Their proof uses the difficult techniques of the $b-$calculus of Melrose and his coauthors, and relies on the very precise description of the geometry of the manifold at infinity, and on the precise decay rate at infinity of the potential $\mathcal{V}$. Thus, there is no hope to extend their proof to more general cases. Moreover, for more general manifolds, it is not clear at all what would be the analog of the exponent $n$ of Guillarmou-Hassell's theorem; in the case of manifolds supporting a global Sobolev inequality and having Euclidean volume growthat infinity, the results of \cite{D1} indicate that the Sobolev exponent should be a good substitute for $n$, but what about more general manifolds? We propose to give an answer to all these questions. We will work in great generality, on manifolds having doubling measure and Gaussian estimates for the heat kernel, which is a natural framewok fo the study of Riesz transforms. Concerning the potential $\mathcal{V}$, we will assume that it decays ``fast enough at infinity'' in a (weak) integral sense. We will prove that under these assumptions, the optimal range of $p$'s for which $d(\Delta+\mathcal{V})^{-1/2}$ is bounded on $L^p$ is $(1,\kappa)$, where $\kappa$ is the {\em parabolic dimension} of $M$. By definition, $\kappa$ is the infimum of all $p$'s such that $M$ is $p$-parabolic. In passing, we prove a novel, very simple characterization of $\kappa$, which is of independent interest. Notice that the recent works \cite{D2}, \cite{Ca}, demonstrate the relevance of the $p$-parabolicity in Riesz transform problems. Our approach also allows us to say something (an alternative $L^p$ inequality), in the case where $d(\Delta+\mathcal{V})^{-1/2}$ is {\em not} bounded on $L^p$ (Theorem \ref{main1}). To the author's knowledge, it is the first time that such an alternative $L^p$ inequality is stated explicitly.

Let us mention that a key point that is used in the proofs of our results is that if $\mathcal{V}$ is ``small at infinity'' in an integral sense, then there exists a function $h$, bounded above and below by positive constant, such that $(\Delta+\mathcal{V})h=0$. That the existence of such a function $h$ has consequences for the Riesz transform $d(\Delta+\mathcal{V})^{-1/2}$ is reminiscent of \cite{GuH}, and also to some extend of \cite{CCH} : indeed, in \cite{CCH}, the unboundedness of the Riesz transform on $L^p$, $p\geq n$ ($p>2$ if $n=2$) on the connected sum $\R^n\sharp\R^n$ of two Euclidean spaces relies on the existence of a non-zero harmonic function with gradient in $L^2$.

\medskip

The plan of this article is as follows: in Section 2, we introduce the setting. In Section 3, we prove a perturbation result for positive solutions of a Schr\"{o}dinger operator. In Section 4, we use the results of Section 3 to prove Gaussian estimates for the heat kernel of Schr\"{o}dinger operators whose potential is ``small at infinity'' in an integral sense. We discuss some consequences for the Riesz transform $d\Delta^{-1/2}$. Section 5 is devoted to prove a characterization of $p$-hyperbolicity, based on volume growth, and which is of independent interest. In Section 6, we introduce a natural scale of weighted $L^p$ spaces, to define an appropriate notion of ``smallness at infinity'' for a potential, in the case where the underlying Riemannian manifold does not satisfy a global Sobolev inequality. Finally, in Section 7, we combine the results of Sections 1-6 to prove --among other results-- a sharp boundedness result for the Riesz transform with potential $d(\Delta+\mathcal{V})^{-1/2}$, for a potential $\mathcal{V}$ ``small at infinity'' in an integral sense.

\section{Preliminaries}

{\em Notation:} for two real functions $g$ and $h$, we will write $g\lesssim h$ if there is a positive constant $C$ such that

$$Cg\leq  h.$$
We also write $g\sim h$ if there exists a positive constant $C$ such that

$$C^{-1}g\leq h\leq Cg.$$

\medskip

{\em Convention: } The Laplacian $\Delta$ is taken with the sign convention that makes it non-negative. For example, on $\R^n$ endowed with the Euclidean metric, $\Delta=-\sum_{k=1}^n\frac{\partial^2}{\partial x_k^2}$.

\bigskip

\subsection{Heat kernel estimates and related inequalities}

 We will always consider $M$ a smooth, complete, connected, non-compact Riemannian manifold.  endowed with a measure $\mu$ which is supposed to be absolutely continuous with respect to the Riemannian measure. We denote by $V(x,r)$ the $\mu$-volume of the geodesic ball $B(x,r)$. The measure $\mu$ is called {\em doubling} if 

\begin{equation}\label{D}\tag{$D$}
V(x,2r)\lesssim V(x,r),\mbox{ for }\mu-\mbox{a.e. }x\in M,\,\forall r>0.
\end{equation}
As a consequence of \eqref{D}, there exist two exponents $\nu$ and $\nu'$ such that

\begin{equation}\label{exp_D}\tag{$D_{\nu,\nu'}$}
\left(\frac{r}{s}\right)^{\nu'}\lesssim \frac{V(x,r)}{V(x,s)} \lesssim\left(\frac{r}{s}\right)^{\nu},
\end{equation}
for all $r \geq s>0$ and $x\in M$. Let us remark that by the Bishop-Gromov volume comparison theorem, if $\mu$ is the Riemannian measure on $M$ and the Ricci curvature on $M$ is non-negative, then one can take $\nu=N$, the topological dimension of $M$, in \eqref{exp_D}. We also introduce the non-collapsing of the volume of balls of radius $1$, which may or may not hold on $M$:

\begin{equation}\label{NC}\tag{$NC$}
1\lesssim V(x,1),\qquad\forall x\in M.
\end{equation}
Thanks to the work of J. Cheeger, M. Gromov and M. Taylor \cite{CGT}, if the Riemann curvature is bounded on $M$ and $\mu$ is the Riemannian measure, then \eqref{NC} is equivalent to a lower bound of the injectivity radius of $M$. In other words, if the Riemann curvature is bounded on $M$, then \eqref{NC} is equivalent to $M$ having bounded geometry. Under milder assumptions on $M$ (for example, if $\mu$ is the Riemannian measure, Ricci curvature bounded from below is enough), \eqref{NC} is equivalent to a family of uniform, local Sobolev inequalities, as we shall explain later. This is a much weaker requirement on $M$, however we shall work in full generality and not assume \eqref{NC} in general.

We will assume sometimes that $(M,\mu)$ satisfies a family of {\em local} Poincar\'{e} inequalities: for every $p\geq 1$, and every $x\in M$, $r>0$, there exists a constant $C_{r}$ such that for all $x\in M$ and all $u\in C_0^\infty(B(x,r))$,

\begin{equation}\label{P_loc}\tag{$P_{loc}$}
\int_{B(x,r)}|u-u_{B(x,r)}|^p\,d\mu\leq C_{r}\int_{B(x,r)}|\nabla u|^p\,d\mu,
\end{equation}
where $u_{B(x,r)}=\frac{1}{V(x,r)}\int_{B(x,r)}u(y)\,d\mu(y).$ As a consequence of \cite{B}, \eqref{P_loc} is satisfied if $\mu$ is the Riemannian measure, and the Ricci curvature is bounded from below on $M$.

\medskip

Let us consider the (weighted) Laplacian $\Delta=\Delta_\mu=-\div(\nabla \cdot)$, where $-\div$ is the formal adjoint of $\nabla$ for the measure $\mu$: for every function $u$ and vector field $X$, both smooth and compactly supported,

$$-\int_M\div(X)u\,d\mu=\int_M(X,\nabla u)\,d\mu.$$
Let us consider the heat semi-group $e^{-t\Delta}$, and its kernel $p_t(x,y)$. Let us introduce on- and off-diagonal estimates for $p_t(x,y)$:

\begin{equation}\label{DUE}\tag{$DUE$}
p_t(x,x)\lesssim \frac{1}{V(x,\sqrt{t})},\qquad \forall x\in M,\,\forall t>0.
\end{equation}
and
\begin{equation}\label{UE}\tag{$UE$}
p_t(x,y)\lesssim \frac{1}{V(x,\sqrt{t})}\exp\left(-c\frac{d^2(x,y)}{t}\right),\qquad \forall (x,y)\in M^2,\,\forall t>0.
\end{equation}
It is a well-known fact (see e.g. \cite{Sik}), using the Gaffney-Davies estimates or equivalently,  the finite speed propagation for the wave equation, that under \eqref{D}, \eqref{DUE} and \eqref{UE} are equivalent. We also consider the two-sided Gaussian estimates (or Li-Yau estimates) for $p_t$: for all $x,y\in M$ and $t>0$,

\begin{equation}\label{LY}\tag{$LY$}
\frac{1}{V(x,\sqrt{t})}e^{-c_1\frac{d^2(x,y)}{t}}\lesssim p_t(x,y)\lesssim \frac{1}{V(x,\sqrt{t})}e^{-c_2\frac{d^2(x,y)}{t}}.
\end{equation}
By the work of A. Grigor'yan and L. Saloff-Coste (see e.g. \cite[Theorem 5.4.12]{SC}), \eqref{LY} is equivalent to the conjunction of \eqref{D} together with the scaled $L^2$ Poincar\'{e} inequalities for the measure $\mu$: 

\begin{equation}\label{P}\tag{$P$}
\int_B|u-u_B|^2\,\mu\lesssim r^2\int_M|\nabla u|^2\,\mu,
\end{equation}
for every ball $B=B(x,r)$ in $M$, and every $u\in C^\infty(B)$; here, $u_B=\frac{1}{\mu(B)}\int_Bu\,d\mu$. Also, it is known from \cite{G4} that \eqref{UE} is equivalent to a family of relative Faber-Krahn inequalities: there exists $\alpha>0$ such that

\begin{equation}\label{RFK}\tag{$RFK$}
\frac{1}{r^2}\left(\frac{V(x,r)}{\mu(\Omega)}\right)^\alpha\lesssim \lambda_1(\Omega),
\end{equation}
for any ball $B(x,r)$ and any relatively compact, open, subset $\Omega\Subset B(x,r)$; here, $\lambda_1(\Omega)$ is the first eigenvalue of $\Delta=\Delta_\mu$ on $\Omega$ with Dirichlet boundary conditions. Finally, let us recall the Sobolev inequality of dimension $n$:

 \begin{equation}\label{S}\tag{$S^n$}
||u||_{\frac{2n}{n-2}}\lesssim ||\nabla u||_2.
\end{equation}
The Sobolev inequality \eqref{S} implies that the following mapping properties for the operators $\Delta^{-\frac{\alpha}{2}}$, as well as the following Gagliardo-Nirenberg type inequalities, hold (see \cite{CSCV}):

\begin{equation}\label{mapping}
\Delta^{-\frac{\alpha}{2}} \,:\,L^p\rightarrow L^q,
\end{equation}
for every $1<p<\frac{n}{\alpha}$ and $\frac{1}{q}=\frac{1}{p}-\frac{\alpha}{n}$; and

\begin{equation}\label{GN}
||u||_\infty\leq C(n,r,s)||\Delta u||^\theta_{r/2}||u||^{1-\theta}_{s/2},\qquad\forall u\in C_0^\infty(M),
\end{equation}
for all $s\geq r>n$, and $\theta=\frac{n/s}{1-(n/r)+(n/s)}$. 

It is well-known that the Sobolev inequality \eqref{S} is related to the volume growth. In fact (see \cite[Theorem 3.1.5]{SC}), if \eqref{S} holds, then

\begin{equation}\label{vol_est}
r^n\lesssim V(x,r),\qquad\forall x\in M,\,\forall r>0,
\end{equation}
which implies that the exponent $n$ in \eqref{S} must be greater or equal to the topological dimension of $M$. This rules out some interesting manifolds that satisfy \eqref{D} and \eqref{UE}: for example, a complete, non-compact manifold with non-negative Ricci curvature (hence, satisfying \eqref{D} and \eqref{UE}), satisfies the Sobolev inequality \eqref{S} if and only if it has {\em maximal} volume growth, i.e. $V(x,r)\sim r^n$, for all $x\in M$, $r>0$ ($n$ being the topological dimension of $M$). An obvious consequence of \eqref{vol_est} is that the Sobolev inequality \eqref{S} implies the non-collapsing of balls \eqref{NC}. Under mild geometric assumptions, the non-collapsing of balls \eqref{NC} is in fact equivalent to a family of uniform, local Sobolev inequalities, as we explain now. Assume that $M$ satisfies \eqref{D} for balls of radius less than $1$, and the heat kernel estimate \eqref{UE} for times less than $1$. For example, this holds if $\mu$ is the Riemannian measure, and the Ricci curvature on $M$ is bounded from below. Then \eqref{NC} implies the following ultracontractivity estimate for small times:

$$||e^{-t\Delta}||_{1,\infty}\lesssim t^{-\frac{\nu}{2}},\qquad\forall t\in (0,1),$$
where $\nu$ is the exponent in \eqref{exp_D}. By the work of Varopoulos (see \cite{CSCV}), this ultracontractivity estimate is equivalent to a family of uniform, local Sobolev inequalities:

\begin{equation}\label{Sloc}\tag{$S^\nu_{loc}$}
||u||_{\frac{2\nu}{\nu-2}}\lesssim ||\nabla u||_2,\qquad\forall u\in C_0^\infty(B),
\end{equation}
for every geodesic ball $B$ of radius less than $1$. Conversely, it is well-known that \eqref{Sloc} implies the volume lower estimate:

$$r^\nu\lesssim V(x,r),\qquad\forall x\in M,\,\forall r\leq 1$$
(see the proof of \cite[Theorem 3.1.5]{SC}). Hence, under \eqref{D} (for balls of small radius) and \eqref{UE} (for small times), \eqref{NC} is equivalent to the family of uniform, local Sobolev inequalities \eqref{Sloc}.

\medskip

Finally, we recall some notions from potential theory. Let $p\in (1,\infty)$, and introduce the (weighted) $p$-Laplacian $\Delta_p(u)=-\div(|\nabla u|^{p-2}\nabla u)$. The notion of $p$-parabolicity of $M$ has recently proved to be important to study the boundedness of the Riesz transform, cf \cite{D2}, \cite{Ca}. It has several equivalent definitions, some of which we now recall. See \cite{CHSC} for more details and references. One definition is that $M$ is $p$-parabolic if and only if every positive supersolution of $\Delta_p$ is constant (Liouville property). An equivalent definition of $p$-parabolicity is that the $p$-capacity ofg every relatively compact, open subset of $M$ is zero. Recall that the $p$-capacity of $U$ is defined as

$$\mathrm{Cap}_p(U)=\inf_u \int_M|\nabla u|^p\,d\mu,$$
where the infimum is taken over all smooth (Lipschitz) functions $u$ with compact support in $M$, such that $u\geq1$ on $U$. If $M$ is not $p$-parabolic, it is said $p$-hyperbolic. In the case $p=2$, the term ``non-parabolic'' is also used. Another characterization of $p$-hyperbolicity of $M$ is the existence of a non-zero (and, in fact, positive) function $\rho$ such that the following $L^p$ Hardy-type inequality holds (see \cite{PT2}):

\begin{equation}\label{Hardy}
\int_M\rho |u|^p\,d\mu\leq \int_M|\nabla u|^p,\qquad\forall u\in C_0^\infty(M).
\end{equation}
It is well-known that volume growth estimates are related to $p$-parabolicity. It is shown in \cite[Corollary 3.2]{CHSC} that for $p\in (1,\infty)$, a necessary condition for $M$ to be $p$-hyperbolic is that for some (all) $x\in M$,

\begin{equation}\label{Vol_p}\tag{$V_p$}
\int_1^\infty \left(\frac{t}{V(x,t)}\right)^{1/p-1}dt<\infty.
\end{equation}
It is also known that in general, \eqref{Vol_p} is not sufficient. However, \eqref{Vol_p} is known to be sufficient if in addition $M$ satisfies \eqref{D} together with scaled $L^p$ Poincar\'{e} inequalities (\cite{H}), or if $M$ has {\em uniform} volume growth and satisfies $L^p$ pseudo-Poincar\'{e} inequalities (\cite[Proposition 3.4]{CHSC}). In particular, if the Ricci curvature is non-negative on $M$, then \eqref{Vol_p} is equivalent to the $p$-hyperbolicity of $M$. It is also true that if $p=2$ and $M$ satisfies \eqref{D} and \eqref{UE}, then ($V_2$) is equivalent to the $2$-hyperbolicity of $M$ (see \cite[Theorem 11.1]{G3}).We now introduce the parabolic dimension of $M$. Let

$$\mathfrak{I}=\{p\in(1,\infty)\,:\,M\mbox{ is }p-\mbox{parabolic}\}.$$
By an observation in \cite[p.1152-1153]{CHSC}, if $M$ satisfies \eqref{P_loc}, then $p$-parabolicity implies $q$-parabolicity for every $q>p$ (we were not aware of this fact in \cite{D2}). In particular, $\mathfrak{I}$ is an interval. Following \cite{C}, let us define

$$\kappa(M)=\inf\mathfrak{I},$$
the {\em parabolic dimension} of $M$ (notice that the term ``hyperbolic dimension'' has been used instead in \cite{D2}). Recall the exponents $\nu$ and $\nu'$ from \eqref{exp_D}, then by the fact that \eqref{Vol_p} is necessary for the $p$-hyperbolicity of $M$, we see that

$$\kappa\leq \nu.$$
We will see later, as a consequence of Theorem \ref{car_hyp} that under \eqref{D}, \eqref{UE} and \eqref{P_loc}, $\kappa\geq \nu'$. Let us highlight these two facts as a Lemma: 

\begin{Lem}\label{para_exp}

Let $M$ satisfying \eqref{D}, \eqref{UE} and \eqref{P_loc}. Recall the exponents $\nu$ and $\nu'$ from \eqref{exp_D}, and let $\kappa$ be the parabolic dimension of $M$. Then

$$\nu'\leq \kappa\leq \nu.$$

\end{Lem}

\subsection{Criticality and perturbation theory for Schr\"{o}dinger operators}

In all the paper, we fix $\{\Omega_k\}_{k=0}^\infty$ an {\em exhaustion} of $M$, i.e. a sequence of smooth, relatively compact domains of $M$ such that $\Omega_0\neq\emptyset$, $\bar{\Omega_k}\subset\Omega_{k+1}$ and

$$M=\bigcup_{k=0}^\infty\Omega_k.$$
We will denote 

$$\Omega_k^*:=M\setminus \Omega_k.$$
We associate to this exhaustion sequence a sequence of smooth cut-off function $\{\chi_k\}_{k=0}^\infty$ such that $\chi_k\equiv 1$ on $\Omega_k$, $\chi_k\equiv 0$ on $\Omega_{k+1}^*$, and $0\leq \chi_k\leq 1$ on $M$. 

We consider $P$ a general Schr\"{o}dinger-type elliptic operator on $M$ in divergence form

\begin{equation}\label{Schro}
Pu=-\div(A\nabla u)+cu,
\end{equation}
where $A$ is locally elliptic and symmetric, and for simplicity $A$ and $c$ are smooth. Let $q$ be the quadratic form associated with $P$, defined by

$$q(u)=\int_M (A\nabla u,\nabla u)d\mu+\int_M cu^2d\mu,\qquad\forall u\in C_0^\infty(M).$$
The positivity and criticality theory of these operators is well-established; we will limit ourselves to quote some results, and refer to the survey \cite{P6} for more details. Let us denote by $\mathcal{C}_P(M)$ the cone of positive solutions of $Pu=0$. $P$ is defined to be {\em non-negative} if its associated quadratic form $q$ is non-negative, and by the celebrated Allegretto-Piepenbrink theorem, this is equivalent to $\mathcal{C}_P(M)\neq \emptyset$. Next, a non-negative $P$ is either {\em subcritical} or {\em critical}, and it is well-known that these two notions have various equivalent definitions; here we define $P$ to be subcritical if it has positive minimal Green functions $G_P(x,y)$. In the case $P=\Delta$, $\Delta$ subcritical is equivalent to $M$ being $2$-hyperbolic. For example the Laplacian on the Euclidean space is subcritical if and only if $n\geq 3$. If $\Omega$ is a regular domain of $M$, we will denote by $G_P^\Omega$ the Green function of $P$ in $\Omega$, with Dirichlet boundary conditions.It is well-known that the sequence $\{G^{\Omega_n}_P\}_{n=0}^\infty$ is increasing, and converges pointwise to $G_P$. Let $\mathcal{V}\in L_{loc}^1$ be a potential. We say that $\mathcal{V}$ is {\em subcritical} for $P$ is $P+\mathcal{V}$ is a subcritical operator. There is a stronger notion, introduced by E.B. Davies and B. Simon \cite{DS}; first denote $\mathcal{V}_+=\max(\mathcal{V},0)$ (resp. $\mathcal{V}_-=\max(-\mathcal{V},0)$) the positive part (resp. negative part) of $\mathcal{V}$. With these notations, $\mathcal{V}_-$ is said {\em strongly subcritical} with respect to $P$ (or $P+\mathcal{V}$ is strongly positive, in the terminology of \cite{CZ}) if there is $\varepsilon\in (0,1]$ such that

\begin{equation}\label{scrit}
\int_M\mathcal{V}_-u^2\leq (1-\varepsilon)\left\{q(u)+\int_M \mathcal{V}_+u^2\right\},\qquad\forall u\in C_0^\infty(M).
\end{equation}
This is equivalent to the inequality $P+\mathcal{V}\geq\varepsilon (P+\mathcal{V}_+)$ at the level of quadratic forms. Notice that if $P$ is subcritical, then the strong subcriticality of $\mathcal{V}_-$ implies that $\mathcal{V}$ is subcritical for $P$. The converse is false in general, but true in some particular cases: for example, it is true if $P=\Delta$ on a manifold satisfying the Sobolev inequality \eqref{S}, and if $\mathcal{V}_-\in L^{\frac{n}{2}}$. Indeed, it is shown in \cite[Definition 6]{D1} that if \eqref{S} holds and $\mathcal{V}_-\in L^{\frac{n}{2}}$, then $\mathcal{V}_-$ is strongly subcritical if and only if $\mathrm{Ker}_{H_0^1}(\Delta+\mathcal{V})=\{0\}$, where $H_0^1$ is the completion of $C_0^\infty(M)$ for the norm $\sqrt{\mathcal{Q}}(u)=\left(\int_M|\nabla u|^2+\mathcal{V}_+u^2\right)^{1/2}$. But, if there exists $\varphi\in \mathrm{Ker}_{H_0^1}(\Delta+\mathcal{V})\setminus\{0\}$, then for every sequence $(\varphi_n)_{n\in \mathbb{N}}$ in $C_0^\infty(M)$, converging to $\varphi$ in $H_0^1$, 

\begin{equation}\label{null1}
\lim_{n\to\infty}\mathcal{Q}(\varphi_n)-\int_M\mathcal{V}_-\varphi_n^2=0.
\end{equation}
Also, by the Sobolev inequality, $H_0^1\hookrightarrow L^{\frac{2n}{n-2}}$, so that 

\begin{equation}\label{null2}
\lim_{n\to\infty}||\varphi_n-\varphi||_{\frac{2n}{n-2}}=0.
\end{equation}
Then, \eqref{null1} and \eqref{null2} imply that $(\varphi_n)_{n\in \mathbb{N}}$ is a {\em null-sequence} (see \cite[Definition 1.1]{PT1}), hence $\Delta+\mathcal{V}$ is critical. Therefore, $\mathcal{V}$ is subcritical if and only if $\mathcal{V}_-$ is strongly subcritical.\\ 

The perturbation theory by a potential for these Schr\"{o}dinger-type operators has been the topic of active research over the past 30 years. Various classes of perturbations have been introduced, in order to prove results such as the stability of the Martin boundary, or the (semi-)equivalence of the Green functions. Actually, both an analytic and a probabilistic approach to the perturbation theory have been developped in parallel. For example, the equivalence of the Green functions of $P$ and of $P+\mathcal{V}$ when both operators are subcritical and $\mathcal{V}$ is a \textit{small perturbation} ($S^\infty(M)$ class in the probabilistic terminology), first proved by analytic means by Pinchover in \cite{P2}, has been later rediscovered with a probabilistic proof by Z.Q. Chen \cite{Ch} and M. Takeda \cite{T2}, building on earlier work by Z. Zhao \cite{Zhao}. However, we warn the reader that the terminology of the perturbation classes in the probabilistic community is often different from the ones in the analytic community. 

We now introduce some  known classes of perturbation, as well as two new ones that are close to some classes introduced by Murata in \cite{M}, and that are tailored to the purposes of this article. Since we are not interested in local regularity issues, from now on the potential $\mathcal{V}$ will be assumed to belong to $L^q_{loc}$ for some $q>n/2$ (where $n$ is the dimension of $M$), which simplifies the definitions given below (indeed, similar definitions can be made in the case where $\mathcal{V}$ is a measure). 

\medskip

Let $P$ be subcritical of the form \eqref{Schro}. We say that $\mathcal{V}$ is a {\em small perturbation} of $P$ if

\begin{equation}\label{small}
\lim_{k\to\infty} \sup_{x,y\in \Omega_{k}^*}\int_{\Omega_k^*} \frac{G_P(x,z)|\mathcal{V}(z)|G_P(z,y)}{G_P(x,y)}\mathrm{d}z=0.
\end{equation}
Small perturbations have first been introduced in \cite{P2}. We say that $\mathcal{V}$ is a {\em G-bounded} perturbation of $P$ if

\begin{equation}\label{G-b}
\sup_{x,y\in M}\int_{M} \frac{G_P(x,z)|\mathcal{V}(z)|G_P(z,y)}{G_P(x,y)}\mathrm{d}z<\infty
\end{equation}
The notion of $G$-bounded perturbation has been introduced in \cite{M}. Of course, if $\mathcal{V}$ is a small perturbation then it is $G-$bounded. Let $h\in \mathcal{C}_P(M)$, i.e. $h$ is a positive solution of $Pu=0$. We say that $\mathcal{V}$ is in the {\em Kato class at infinity} with respect to $(P,h)$, denoted $K^\infty(M,P,h)$, if

\begin{equation}\label{Kato}
\lim_{k\to\infty}\sup_{x\in \Omega_k^*}\int_{\Omega_k^*}\frac{G_P(x,y)|\mathcal{V}(y)|h(y)}{h(x)}dy=0.
\end{equation}
In the case $P=\Delta$ and $h\equiv\mathbf{1}$, we will simply speak of the Kato class at infinity of $M$, denoted $K^\infty(M)$. For example, if $P=\Delta$, $M=\R^n$ for $n\geq3$, and for $|x|\geq A$, 

$$|\mathcal{V}(x)|\leq \frac{\varphi(|x|)}{|x|^2},$$
where $\varphi$ is a non-increasing, continuous function which satisfies

$$\int_A^\infty \frac{\varphi(s)}{s}\,ds<\infty,$$
then $\mathcal{V}\in K^\infty(\R^n)$ (see \cite[Lemma 2.3]{P4}).

More generally, we introduce the following new definition: for $\varepsilon>0$, we will say that $\mathcal{V}$ satisfies the condition $(K^\infty,M,P,h,\varepsilon)$ if 

\begin{equation}\label{Kato-eps}
\lim_{k\to\infty}\sup_{x\in \Omega_k^*}\int_{\Omega_k^*}\frac{G_P(x,y)|\mathcal{V}(y)|h(y)}{h(x)}dy<\varepsilon.
\end{equation}
When $P=\Delta$ and $h\equiv\mathbf{1}$, we will simply speak of the condition $(K^\infty,\varepsilon)$. By the Maximum Principle, in \eqref{small}, \eqref{Kato} and \eqref{Kato-eps}, one can replace the supremum over $\Omega_k^*$ by the supremum over $M$ (see \cite[Lemma 2.1]{M}). Finally, we introduce a notion closely related to the $H-$boundedness introduced by M. Murata in \cite{M}: for a positive solution $h$ of $Pu=0$, we say that $\mathcal{V}$ is $(H,h)-${\em bounded} if

\begin{equation}\label{H-b}
\sup_{x\in M}\int_{x\in M}\frac{G_P(x,y)|\mathcal{V}(y)|h(y)}{h(x)}dy<\infty.
\end{equation}
If $\mathcal{V}$ is $(H,h)-$bounded, then we define

$$||\mathcal{V}||_{H,h}:=\sup_{x\in M}\int_{x\in M}\frac{G_P(x,y)|\mathcal{V}(y)|h(y)}{h(x)}dy<\infty.$$
Obviously, if $\mathcal{V}$ satisfies condition $K^\infty(M,P,h,\varepsilon)$ for some $\varepsilon>0$, then $\mathcal{V}$ is $(H,h)-$bounded. In particular, if $\mathcal{V}$ is in the Kato class at infinity $K^\infty(M,P,h)$, then $\mathcal{V}$ is $(H,h)-$bounded. Furthermore, it is a direct application of Martin's theory that if $\mathcal{V}$ is a small (resp. $G-$bounded) perturbation of $P$, then for every $h\in \mathcal{C}_P(M)$, $\mathcal{V}$ is in $K^\infty(M,P,h)$ (resp., $\mathcal{V}$ is $(H,h)-$bounded) (see \cite{P1}, \cite{P2}, \cite{M}).\\

We warn again the reader that equivalent classes may be found in the literature, under different names. To conclude this discussion, let us give a particular but important example of potentials in $K^\infty(M)$:

\begin{Exa}\label{Sob_Kato}
{\em
Assume that $M$ satisfies the Sobolev inequality \eqref{S}, and let $\mathcal{V}\in L^{\frac{n}{2}\pm\varepsilon}$, for some $\varepsilon>0$. Then $\mathcal{V}\in K^\infty(M)$.
}
\end{Exa}

\begin{Rem}
{\em
This example will be generalized later (in Proposition \ref{cond_Kato}) to manifolds which satisfy only \eqref{D} and \eqref{UE}, but not the Sobolev inequality \eqref{S}.
}
\end{Rem}
\begin{proof}

Let 

$$u(x)=\int_MG(x,y)|\mathcal{V}(y)|dy,$$
that is, $u=\Delta^{-1}|\mathcal{V}|$. Then by the fact that $\mathcal{V}\in L^{\frac{n}{2}-\varepsilon}$ and \eqref{mapping}, there is $s>n$ defined by $\frac{2}{s}=\frac{1}{\frac{n}{2}-\varepsilon}-\frac{2}{n}$, such that $u\in L^{s/2}$. Also, $\Delta u=\mathcal{V}\in L^{r/2}$ with $\frac{r}{2}=\frac{n}{2}+\varepsilon$, therefore by \eqref{GN}, we deduce that 

\begin{equation}\label{est_inf}
||u||_\infty\leq C(n,\varepsilon)||\mathcal{V}||_{\frac{n}{2}+\varepsilon}^{\theta}||u||^{1-\theta}_{\frac{s}{2}} \leq C(n,\varepsilon)||\mathcal{V}||_{\frac{n}{2}+\varepsilon}^{\theta}||\mathcal{V}||^{1-\theta}_{\frac{n}{2}-\varepsilon}.
\end{equation}
Let $\mathcal{V}_k=\mathcal{V}\chi_{k}$, then 

$$\lim_{R\to\infty}||\mathcal{V}_k||_{\frac{n}{2}\pm\varepsilon}=0.$$
Applying \eqref{est_inf} with $\mathcal{V}_k$ instead of $\mathcal{V}$, and letting $k\to\infty$, we deduce that

$$\lim_{k\to\infty} \sup_{x\in M}\int_{\Omega_k^*} G(x,y)|\mathcal{V}(y)|dy=0,$$
i.e. $\mathcal{V}$ belongs to $K^\infty(M)$.

\end{proof}

\subsection{$h$-transform}

We recall a standard procedure to eliminate the zero-order term of an operator $P$ of the form \eqref{Schro}. Let $h \in \mathcal{C}_{P}(M)$, and define a map
\begin{equation}
\label{eq:gmap}
T_h: v \to hv\,.
\end{equation}
Notice that $T_h$ is an isometry between $L^2(M,\, h^2d\nu)$ and $L^2(M,\, d\nu)$. The operator $P_h:=T_h^{-1} \circ P \circ T_h$, that is,
\begin{equation}
\label{htransform}
P_h u = \frac{P(h u)}{h}
\end{equation}
is called the {\em $h$-transform} (or {\em Doob transform}) of $P$. Notice that

\begin{equation*}
P_h \mathbf{1}=0 .
\end{equation*}
Also, it is not hard to see that $P_h$ is explicitly given by
\begin{equation}\label{Ph}
P_h u=-\frac{1}{h^{2}}\mathrm{div}(h^2 A(x)\nabla u),
\end{equation}
and $P_h$ is self-adjoint on $L^2(M,\,h^{2}\mathrm{d} \nu)$. Moreover, $P$ and $P_h$ are unitary equivalent. It is also easy to prove that $P_h$ is subcritical if and only if
$P$ is subcritical, and in this case, the corresponding Green function satisfies
$$G_{P_h}(x,y)=\frac{G_{P}(x,y)}{h(x)h(y)}.$$
On the other hand, in the critical case $\mathbf{1}$ is the ground state of the equation $P_h u=0$ in $M$. Finally, notice that the use of the $h-$transform allows to rewrite \eqref{Kato-eps}, i.e. the condition $(K^\infty,P,h,\varepsilon)$, as

$$\lim_{k\to\infty} ||P_h^{-1}|\mathcal{V}|||_{L^\infty(\Omega_k^*)\to L^\infty(\Omega_k^*)}<\varepsilon.$$

\section{Perturbation result for positive solutions of a Schr\"{o}dinger operator}

In \cite{M}, Murata introduced the class of {\em semi-small perturbations}, and proved that if $\mathcal{V}$ is a semi-small perturbation of a subcritical operator $P$, then the minimal Martin boundaries of $P$ and $P+\mathcal{V}$ are homeomorphic. By the Martin representation theorem, every positive solution $h$ of $Pu=0$ in $M$ can be written as 

$$h(x)=\int_{\partial_m (M,P)}K(x,\xi)\,d\nu(\xi),$$
for some probability measure $d\nu$ on the minimal Martin boundary $\partial_m (M,P)$. Here, $K(x,\xi)$ denotes the Martin kernel. This implies that if $\mathcal{V}$ is a semi-small perturbation of $P$, then there is a bijection that preserves order from the cones of positive solutions of $P$, $\mathcal{C}_P(M)$, into $\mathcal{C}_{P+\mathcal{V}}$. In this section, we will be concerned with the following related problem:

\begin{Prob}\label{Q}

Let $P$ in the form \eqref{Schro} be subcritical. If $h$ is a positive solution of $Pu=0$ in $M$, under which conditions on $\mathcal{V}$ does there exist $g\sim h$ such that $(P+\mathcal{V})g=0$? 

\end{Prob}
In the case $P=\Delta$ and $h\equiv\mathbf{1}$, an answer to Problem \ref{Q} can be extracted from Takeda's article \cite{T1}. Takeda's arguments are of probabilist nature; the main idea goes back to the pioneering work of B. Simon \cite{S}, who first proved by a probabilistic argument that on $\R^n$, $n\geq3$, if $\mathcal{V}\in L^{\frac{n}{2}-\varepsilon}\cap L^{\frac{n}{2}+\varepsilon}$ then the existence of $g\sim \mathbf{1}$ solution of $(\Delta+\mathcal{V})u=0$, is equivalent to $\Delta+\mathcal{V}$ being subcritical. Problem \ref{Q} was also studied by Pinchover, who solved it for small perturbations (see \cite[Lemma 2.4]{P1}, \cite[Lemma 1.1]{P2}). In this section, we will present an answer to Problem \ref{Q}, more precisely we will give an analytic proof of the following result:

\begin{Thm}\label{pos_sol}

Let $P$ be subcritical. Let $h$ be a positive solution of $Pu=0$ in $M$, and let $\mathcal{V}$ be a potential such that $\mathcal{V}_-$ satisfies $(K^\infty,P,h, 1)$ and $\mathcal{V}_+$ is $(H,h)-$bounded. Assume that $P+\mathcal{V}$ is subcritical. Then there exists $g\sim h$, positive solution of $(P+\mathcal{V})u=0$. Furthermore, $g$ satisfies

$$g(x)=h(x)-\int_MG_P(x,y)\mathcal{V}(y)g(y)dy.$$

\end{Thm}

\begin{Rem}
{\em
In particular, Theorem \ref{pos_sol} applies if $\mathcal{V}$ is in the Kato class at infinity $K^\infty(M,P,h)$.
}
\end{Rem}

\begin{Rem}
{\em
In fact, as the proof will show, the following lower estimate of $g$ holds (compare with Equation (3.2) in \cite{T1}):

$$e^{-||\mathcal{V}_+||_{H,h}}\leq \frac{g}{h}.$$
In \cite{T1}, in the case $P=\Delta$, $h\equiv \mathbf{1}$ and under the extra assumption that $\mathcal{V}_-$ is strongly subcritical with respect to $\Delta+\mathcal{V}$, an  upper-bound --with a probabilistic flavor-- for $g$ is given:

$$\frac{g}{h}\leq \sup_{x\in M}\mathbf{E}_x\exp\left(-\int_0^\infty \mathcal{V}(B_s)\mathrm{d}s\right),$$
where $B_s$ is the Brownian motion on $M$, and $\mathbf{E}_x$ is the conditional expectation, starting from $x$. 
 
}
\end{Rem}
In the case $P=\Delta$ and $h\equiv 1$, and $\mathcal{V}_-$ strongly subcritical with respect to $\Delta+\mathcal{V}_+$, Theorem \ref{pos_sol} follows from Theorem 1, Equation (3.2), as well as Lemma 2, in \cite{T1}. It was not noticed in \cite{T1} that the strong subcriticality of $\mathcal{V}_-$ can be replaced by the weaker assumption that $\Delta+\mathcal{V}$ is subcritical. Also, our assumption that $\mathcal{V}_-$ satisfies $(K^\infty,P,h,1)$ is weaker than the assumption that $\mathcal{V}_-$ belongs to $K^\infty(M,P,h)$. In the case where $\mathcal{V}$ is a small perturbation of $P$ (which is a much stronger condition), Theorem \ref{pos_sol} follows from \cite[Lemma 2.4]{P1} and \cite[Lemma 2.4]{P2}.  

Let us start with the following lemma, which is essentially well-known (see \cite[Lemma 3.3]{P3}), but whose proof is provided since it will be instrumental in the proof of Theorem \ref{pos_sol}:

\begin{Lem}\label{Vsmall}

Assume that $\mathcal{V}$ is $(H,h)-$bounded and that

$$||\mathcal{V}||_{H,h}<\frac{1}{2}.$$
Then there exists $g\sim h$ solution of $(P+\mathcal{V})u=0$. Furthermore, if $\mathcal{V}$ is non-positive, then the existence of $g$ is guaranteed as soon as

$$||\mathcal{V}||_{H,h}<1.$$

\end{Lem}
For the convenience of the reader, we give the proof of Lemma \ref{Vsmall}:\\

\noindent {\em Proof of Lemma \ref{Vsmall}:} let 

$$\varepsilon:=\sup_{x\in M}\int_M \frac{G_P(x,y)|\mathcal{V}(y)|h(y)}{h(x)}dy<1.$$
We want to define $g$ by the formula

$$g=(I+P^{-1}\mathcal{V})^{-1}h.$$
In fact, let us define $g$ by the Neumann series

$$g=\sum_{k=0}^\infty (-1)^k(P^{-1}\mathcal{V})^kh.$$
If the series converges, then it is easy to see that $g$ is solution of $(P+\mathcal{V})u=0$. Define $h_k:=(P^{-1}\mathcal{V})^kh$. Then,

$$h_k(x)=\int_MG_P(x,y)\mathcal{V}(y)h_{k-1}(y)dy,$$
and by an easy induction,

$$|h_k|\leq \varepsilon^{k}h.$$
Hence

$$\left(1-\sum_{k=1}^\infty \varepsilon^{k}\right)h\leq g\leq \left(\sum_{k=0}^\infty \varepsilon^{k}\right)h,$$
that is

$$\frac{1-2\varepsilon}{1-\varepsilon}h\leq g\leq \frac{1}{1-\varepsilon}h,$$
hence the result in the general case. In the case where $\mathcal{V}$ is non-positive, then $P^{-1}\mathcal{V}\leq 0$, which implies that $h\leq g$ from the definition of $g$. Thus, by the previous computation,

$$h\leq g\leq \frac{1}{1-\varepsilon}h,$$
hence $g\sim h$ as soon as $\varepsilon<1$.

\cqfd

\medskip

\noindent{\em Proof of Theorem \ref{pos_sol}: } we split the proof into two parts.

\medskip

\noindent {\em Step 1: case $\mathcal{V}\geq0$}.

\medskip

Without loss of generality, one can assume that $\mathcal{V}\neq 0$, that is $||\mathcal{V}||_{H,h}>0$. For $t\geq 0$, define

$$g_t(x)=h(x)-t\int_MG_{P+t\mathcal{V}}(x,y)\mathcal{V}(y)h(y)\,dy.$$
We will employ the following lemma:

\begin{Lem}\label{conv}

For every $t\geq 0$, $g_t$ is a positive solution of $(P+t\mathcal{V})u=0$. Furthermore, let $0\leq t_0<t_1<t_2<\infty$, and define $\alpha\in (0,1)$ by

$$t_1=(1-\alpha)t_0+\alpha t_2.$$
Then

$$g_{t_1}\leq g_{t_0}^{1-\alpha}g_{t_2}^\alpha.$$

\end{Lem}
Assuming for the moment the result of Lemma \ref{conv}, let us finish the proof of Step 1. We apply Lemma \ref{conv} with $t_0=0$, $t_1=\varepsilon$ and $t_2=t>\varepsilon$. It yields

$$g_t^{\varepsilon/t}\geq g_{\varepsilon} h^{-1+\varepsilon/t}.$$
Since $\mathcal{V}\geq0$, one has $G_{P+t\mathcal{V}}\leq G_P$, therefore, for all $x\in M$,

\begin{eqnarray}
\int_MG_{P+t\mathcal{V}}(x,y)\mathcal{V}(y)h(y)\,dy&\leq &\int_MG_P (x,y)\mathcal{V}(y)h(y)\,dy\nonumber\\
&\leq &||\mathcal{V}||_{H,h}\,h(x).\nonumber
\end{eqnarray}
Consequently, if $C_\varepsilon=1-\varepsilon ||\mathcal{V}||_{H,h}$, one has for $\varepsilon<||\mathcal{V}||_{H,h}^{-1}$,

$$C_{\varepsilon} h\leq g_{\varepsilon}\leq h.$$
Thus

$$g_t\geq C_{\varepsilon}^{t/\varepsilon} h=e^{-C'_{\varepsilon}t}h,$$
where $C'_{\varepsilon}=-\varepsilon^{-1}\log(1-\varepsilon||\mathcal{V}||_{H,h})$. Letting $\varepsilon\to0$, one gets

$$g_t\geq e^{-t||\mathcal{V}||_{H,h}}h.$$
Applying this for $t=1$ and defining $g=g_1$, we find that $g$ is a positive solution of $(P+\mathcal{V})u=0$ such that

$$e^{-||\mathcal{V}||_{H,h}}\leq \frac{g}{h}\leq 1.$$
This concludes the proof in the case where $\mathcal{V}\geq0$.

\medskip

\noindent {\em Step 2: general case}.

\medskip

Since $\mathcal{V}_-$ satisfies condition $(K^\infty,P,h,1)$, we can fix $k\in\mathbb{N}$ such that 

$$\sup_{x\in M}\int_{\Omega_k^*}\frac{G_P(x,y)\mathcal{V}_-(y)h(y)}{h(x)}\,dy<1.$$
Define

$$\mathcal{V}_{-,0}=\mathcal{V}_-\chi_k,\,\,\mathcal{V}_{-,\infty}=\mathcal{V}_--\mathcal{V}_{-,0}.$$
Notice that $\mathcal{V}_{-,0}$ has compact support, and that  $||\mathcal{V}_{-,\infty}||_{H,h}<1$. By Step 1, there exists $g_1\sim h$, solution of $(P+t\mathcal{V}_+)u=0$, and $g_1$ is given by

$$g_1=h-(P+\mathcal{V}_+)^{-1}\mathcal{V}_+h.$$ 
Since $G_{P+\mathcal{V_+}}\leq G_P$, 

$$\sup_{x\in M}\int_{M}\frac{G_{P+\mathcal{V}_+}(x,y)\mathcal{V}_{-,\infty}(y)h(y)}{h(x)}\,dy<1.$$
By Lemma \ref{Vsmall}, there exists $g_2\sim g_1$ solution of $(P+\mathcal{V}_+-\mathcal{V}_{-,\infty})u=0$, and $g_2$ is given by

$$g_2=(I-(P+\mathcal{V}_+)^{-1}\mathcal{V}_{-,\infty})^{-1}g_1.$$
Define 

$$g=g_2+(P+\mathcal{V})^{-1}\mathcal{V}_{-,0}\,g_2.$$
Obviously, $g$ satisfies $(P+\mathcal{V})g=0$, and $g_2\leq g$. We are going to show that $g\leq Cg_2$, and for this purpose it is clearly enough to show that there is a constant $C$ such that for every $x\in M$,

\begin{equation}\label{upb}
\int_MG_{P+\mathcal{V}}(x,y)\mathcal{V}_{-,0}(y)g_2(y)\,dy\leq Cg_2(x). 
\end{equation}
Denote $L=P+\mathcal{V}_+-\mathcal{V}_{-,\infty}$. By \cite{P1}, since $\mathcal{V}_{-,0}$ has compact support, there is a constant $C$ such that

$$C^{-1}G_{P+\mathcal{V}}\leq G_{L}  \leq CG_{P+\mathcal{V}}.$$
Consequently,

$$\int_MG_{P+\mathcal{V}}(x,y)\mathcal{V}_{-,0}(y)g_2(y)\,dy\leq C\int_MG_{L}(x,y)\mathcal{V}_{-,0}(y)g_2(y)\,dy.$$
Denote $f(y)=\mathcal{V}_{-,0}(y)g_2(y)$, then $f$ is non-negative and has support included in $\Omega_k$. Notice that $u(x)=\int_MG_{L}(x,y)f(y)\,dy$ is a positive solution of $Lv=f$ on $M$, and

$$u(x)=\lim_{n\to\infty}u_n,$$
with

$$u_n=\int_MG_{L}^{\Omega_n}(x,y)f(y)\,dy.$$
Notice that for $n\geq k$, $u_n$ is solution of $Lu=0$ in $\Omega_n\setminus \Omega_k$, and vanishes identically on the boundary of $\Omega_n$. Therefore, by the Maximum Principle, there exists a constant 

$$C_n=\frac{\sup_{\partial \Omega_k}g_2}{\inf_{\partial \Omega_k}u_n}$$
such that

$$u_n(x)\leq C_ng_2(x),\qquad \forall x\in \Omega_n\setminus \Omega_k.$$
Notice that $(u_n)_{n=0}^\infty$ is increasing, so that $C_n\leq C_{k}$. Letting $n\to\infty$, one deduces that

$$u(x)\leq C_kg_2(x),\qquad \forall x\in \Omega_k^*.$$
Since $u$ and $g_2$ are positive and continuous on $\Omega_k$, one can increase $C_k$ to ensure that

$$u(x)\leq Cg_2(x),\qquad \forall x\in M.$$
Thus, \eqref{upb} holds. This implies that

$$g\sim g_2,$$
and since $g_1\sim h$ and $g_2\sim g_1$, one concludes that

$$g\sim h.$$
In order to conclude the proof of Theorem \ref{pos_sol}, one has to prove that $g$ satisfies the equation

\begin{equation}\label{int}
g(x)=h(x)-\int_MG_P(x,y)\mathcal{V}(y)g(y)dy.
\end{equation}
For $n\geq N$, denote by $h_n$ the restriction of $h$ to $\Omega_n$, and define $g_{1,n}$, $g_{2,n}$ and $g_n$ by

$$g_{1,n}(x)=h_n(x)-\int_{\Omega_n}G_{P+\mathcal{V}_+}^{\Omega_n}(x,y)\mathcal{V}_+(y)h_n(y)\,dy,$$ 

$$g_{2,n}=(I-T_n)^{-1}g_{1,n},$$
where $T_n$ is the operator

$$T_nu(x)=\int_{\Omega_n}G_{P+\mathcal{V}_+}^{\Omega_n}(x,y)\mathcal{V}_{-,\infty}(y)u(y)\,dy,$$
and

$$g_n(x)=g_{2,n}(x)+\int_{\Omega_n}G_{P+\mathcal{V}}^{\Omega_n}(x,y)\mathcal{V}_{-,0}(y)g_{2,n}(y)\,dy.$$
It follows from the hypotheses that there is a constant $C$ such that for every $n\geq N$ and $i=1,2$,

$$C^{-1}h\leq g_{i,n}\leq Ch,$$
and

$$C^{-1}h\leq g_n\leq Ch.$$
Since $\mathcal{V}$ is $(H,h)$-bounded, the Dominated Convergence Theorem implies that pointwise,

$$\lim_{n\to \infty}g_{i,n}=g_i,\,i=1,2,$$
and

$$\lim_{n\to \infty}g_n=g.$$
Define

$$w_n(x)=h_n(x)-\int_{\Omega_n}G_P^{\Omega_n}(x,y)\mathcal{V}(y)g_n(y)dy,$$
and notice that by the fact that $\mathcal{V}$ is $(H,h)$-bounded and the Dominated Convergence Theorem, as $n\to \infty$, $w_n$ converges pointwise to 

$$w(x)=h(x)-\int_MG_P(x,y)\mathcal{V}(y)g(y)dy.$$
One wishes to show that $w=g$. Notice that for every $n\geq N$, $w_n$ and $g_n$ are solutions of the following Dirichlet problem:

$$\left\{\begin{array}{rcl}

Pu&=&-\mathcal{V}g_n\mbox{ in }\Omega_n\\
u|_{\partial\Omega_n}&=&h|_{\partial\Omega_n}

\end{array}\right.$$
The Maximum Principle implies that

$$w_n=g_n.$$
Passing to the limit as $n\to\infty$ gives

$$w=g,$$
that is, \eqref{int}.

\cqfd

\noindent {\em Proof of Lemma \ref{conv}:}

We first prove the inequality

\begin{equation}\label{econv}
g_{t_1}\leq g_{t_0}^{1-\alpha}g_{t_2}^\alpha.
\end{equation}
Let $\{\Omega_n\}_{n=0}^\infty$ be an exhaustion of $M$, and for $t\geq0$, define

$$g_{t,n}(x)=h|_{\Omega_n}(x)-t\int_{M}G^{\Omega_n}_{P+t\mathcal{V}}(x,y)\mathcal{V}(y)h(y)\,dy.$$
Since $\{G^{\Omega_n}_{P+t\mathcal{V}}\}_{n\in\mathbb{N}}$ is non-increasing and converges pointwise to $G_{P+t\mathcal{V}}$ as $n\to \infty$, the Monotone Convergence Theorem implies that $\{g_{t,n}\}_{n=0}^\infty$ converges pointwise to $g_t$ as $n\to\infty$. Let us remark that $g_{t,n}$ is solution of the following Dirichlet problem in $\Omega_n$:

$$\left\{
\begin{array}{rcl}
(P+t\mathcal{V})u&=&0\mbox{ in }\Omega_n, \\
u|_{\partial\Omega_n}&=&h|_{\partial\Omega_n}.
\end{array}\right.$$
Let us consider 

$$u_n=g_{n,t_0}^{1-\alpha}g_{n,t_2}^\alpha.$$
By an easy computation (see the proof of \cite[Theorem 3.1]{P5} or \cite[Lemma 5.1]{DFP}), $u_n$ is a supersolution of $(P+t_1\mathcal{V})$ in $\Omega_n$, and is equal to $h$ on the boundary of $\Omega_n$. Therefore, by the Maximum Principle, 

$$u_n \geq g_{n,t_1},$$
that is

$$g_{n,t_1}\leq g_{n,t_0}^{1-\alpha}g_{n,t_2}^\alpha.$$
Letting $n\to\infty$, one finds \eqref{econv}. Let us show that \eqref{econv} implies that $g_t$ is positive for all $t>0$. Take $0<\varepsilon<||\mathcal{V}||_{H,h}^{-1}$, then 

$$C_\varepsilon h\leq g_\varepsilon\leq h,$$
where $C_\varepsilon=1-\varepsilon||\mathcal{V}||_{H,h}>0$. By \eqref{econv} with $t_0=0$, $t_1=\varepsilon$ and $t_2=t$, one has

$$g_\varepsilon\leq h^{1-\varepsilon/t}g_t^{\varepsilon/t},$$
therefore

$$C_\varepsilon^{t/\varepsilon}h\leq g_t,$$
which implies that $g_t$ is positive. This concludes the proof.

\cqfd

To conclude this section, we give a simple analytic proof of a result of Takeda \cite{T2} and Chen \cite{Ch}; it provides conditions under which the strong subcriticality of $\mathcal{V}_-$ with respect to $P+\mathcal{V}_+$ is equivalent to the subcriticality of $P+\mathcal{V}$:

\begin{Thm}\label{strong_sub}

Let $P$ be subcritical of the form \eqref{Schro}. Assume that $\mathcal{V}$ is a potential such that $\mathcal{V}_-$ is a small perturbation of $P$ and $\mathcal{V}_+$ is $G-$bounded. Then, $P+\mathcal{V}$ is subcritical if and only if $\mathcal{V}_-$ is strongly subcritical with respect to $P+\mathcal{V}_+$.

\end{Thm}

\begin{proof}

The proof relies on the following:

\begin{Lem}\label{equiv_Green}

Let $\mathcal{V}_-$ be a small perturbation of $P$ and $\mathcal{V}_+$ is $G-$bounded with respect to $P$, and assume that $L=P+\mathcal{V}$ is subcritical. Then the Green function $G_P$ of $P$ is equivalent to the Green function $G_L$ of $L$.

\end{Lem}
Lemma \ref{equiv_Green} follows from \cite[Lemma 2.4]{P2} and \cite[Corollary 3.6]{P5}. Let $h$ be a positive solution of $Pu=0$. Since $\mathcal{V}_-$ is a small perturbation of $P$, in particular it is $(H,h)$-bounded:

$$\sup_{x\in M}\int_MG_P(x,y)\mathcal{V}_-(y)h(y)\,dy<\infty.$$
By Lemma \ref{equiv_Green}, it follows that

$$\sup_{x\in M}\int_MG_L(x,y)\mathcal{V}_-(y)h)(y)\,dy<\infty.$$
According to Lemma \ref{Vsmall}, for $\varepsilon<||\mathcal{V}||_{H,\mathbf{1}}^{-1}$, there is $g\sim h$ (in particular, positive) solution of $(L-\varepsilon \mathcal{V}_-)u=0$. By the Allegretto-Piepenbrink Theorem, it follows that $L-\varepsilon \mathcal{V}_-$ is non-negative, i.e. the inequality

$$\varepsilon\int_M\mathcal{V}_-u^2\leq q(u)+\mathcal{V}_+u^2-\mathcal{V}_-u^2,\qquad \forall u\in C_0^\infty(M),$$
is satisfied, where $q$ is the quadratic form of $P$. Equivalently,

$$\int_M\mathcal{V}_-u^2\leq (1+\varepsilon)^{-1}\left\{q(u)+\mathcal{V}_+u^2\right\},\qquad \forall u\in C_0^\infty(M).$$
This shows that $\mathcal{V}_-$ is strongly subcritical.

\end{proof}

\section{Heat kernel estimates}

In the section, we give consequences of Theorem \ref{pos_sol} for the heat kernel of $\Delta+\mathcal{V}$. Let us denote by $p_t^{\mathcal{V}}$ the heat kernel of $\Delta+\mathcal{V}$. We introduce the Gaussian upper-estimate for $p_t^{\mathcal{V}}(x,y)$:

\begin{equation}\label{UEV}\tag{$UE_{\mathcal{V}}$}
|p_t^{\mathcal{V}}(x,y)|\leq \frac{C}{V(x,\sqrt{t})}e^{-c\frac{d^2(x,y)}{t}},\qquad \forall (x,y)\in M^2,\,\forall t>0.
\end{equation}
We also introduce the two-sided Gaussian (Li-Yau) estimates for $p_t^{\mathcal{V}}(x,y)$: for every $(x,y)\in M^2$ and $t>0$,

\begin{equation}\label{LYV}\tag{$LY_{\mathcal{V}}$}
\frac{C_1}{V(x,\sqrt{t})}e^{-c_1\frac{d^2(x,y)}{t}}\leq p_t^{\mathcal{V}}(x,y)\leq \frac{C_2}{V(x,\sqrt{t})}e^{-c_2\frac{d^2(x,y)}{t}}.
\end{equation}

\begin{Thm}\label{heat_kernel}

Let $(M,\mu)$ be a non-parabolic weighted manifold with heat kernel satisfying \eqref{D} and \eqref{UE}. Let $\mathcal{V}$ be a subcritical potential, such that $\mathcal{V}_-$ satisfies the condition $(K^\infty,1)$ and $\mathcal{V}_+$  is $(H,\mathbf{1})-$bounded. Then the Gaussian upper-estimate \eqref{UEV} for the heat kernel of $\Delta_\mu+\mathcal{V}$ holds. If moreover the scaled Poincar\'{e} inequalities \eqref{P} are satisfied, then the full Li-Yau estimates \eqref{LYV} hold for the heat kernel of $\Delta_\mu+V$.

\end{Thm}

\begin{Rem}
{\em 

In fact, using domination theory, it will be apparent from the proof that Theorem \ref{heat_kernel} holds under the following weaker assumption on $\mathcal{V}_+$: it is enough to assume that $\mathcal{V}_+\geq W$ for some $(H,\mathbf{1})-$bounded potential $W$, such that $W-\mathcal{V}_-$ is subcritical. In the forthcoming paper \cite{CDS}, under the stronger assumption that $\mathcal{V}_-$ is strongly subcritical, Theorem \ref{heat_kernel} is proved with no assumption on $\mathcal{V}_+$. However, we do not know if the subcriticality of $\mathcal{V}$ is enough in general.

}
\end{Rem}
In the case where \eqref{P} is satisfied, Theorem \ref{heat_kernel} has been proved in \cite[Theorem 2]{T1} under the stronger assumptions that $\mathcal{V}_-\in K^\infty(M)$ and is strongly subcritical with respect to $\Delta+\mathcal{V}_+$. As we have already mentionned for Theorem \ref{pos_sol}, the result of Theorem \ref{heat_kernel} holds under much weaker local regularity hypotheses on $\mathcal{V}$, but we will not pursue this here. Let us also mention that under \eqref{P}, Theorem \ref{heat_kernel} has first been proved in \cite[Theorem 10.5]{G} under the assumption that $\mathcal{V}\geq0$. Also, the idea that Theorem \ref{heat_kernel} follows from Theorem \ref{pos_sol} originally comes from \cite{G}; the proof uses a very useful device traditionally called the {\em Doob transform} or {\em h-transform} with respect to a positive solution $h$ of $(\Delta+\mathcal{V})h=0$, and that allows one to pass from a Schr\"{o}dinger operator $\Delta+\mathcal{V}$ to the weighted Laplacian $\Delta_{h^2\mu}$. In our proof of Theorem \ref{heat_kernel}, we will use the same idea, and indeed, the only new element in our proof of Theorem \ref{heat_kernel} is that one can treat upper-bounds of the heat kernel {\em only}, using the caracterization of the Gaussian upper-estimates for a (weighted) Laplacian in term of relative Faber-Krahn inequalities.

\medskip

\noindent{\em Proof of Theorem \ref{heat_kernel}: }let $P=\Delta+\mathcal{V}$. By Theorem \ref{pos_sol}, there exists $h\sim 1$ solution of $Pu=0$. Let us consider the $h$-transform $P_h$:

$$P_h=h^{-1}(\Delta+\mathcal{V})h,$$
which is self-adjoint on $L^2(\Omega,h^2\,d\mu)$. By \eqref{Ph}, the operator $P_h$ is nothing but the weighted Laplacian $\Delta_{h^2\mu}$. Furthermore, the heat kernel $p_t^h(x,y)$ of $P_h$ on $L^2(\Omega,h^2\, d\mu)$ is given by

$$p_t^h(x,y)=\frac{p_t^{\mathcal{V}}(x,y)}{h(x)h(y)}.$$
Since $h\sim 1$, it is thus enough to prove the Gaussian upper-estimates (resp. the two-sided Gaussian estimates) for $p^h_t(x,y)$. Let us start by proving the upper-bound. Since \eqref{D} and \eqref{UE} hold, $(M,\mu)$ satisfies the  relative Faber-Krahn inequalities \eqref{RFK}. Since $h\sim 1$, it follows that $(M,h^2\mu)$ also satisfies the relative Faber-Krahn inequality \eqref{RFK}. Therefore, by \cite{G4}, we conclude that the heat kernel of $P_h=\Delta_{h^2\mu}$ has Gaussian upper-estimates. The case where $(M,\mu)$ satisfies the scaled $L^2$ Poincar\'{e} inequalities \eqref{P} follows the same idea: according to the work of Grigor'yan and Saloff-Coste (see e.g. \cite[Theorem 5.4.12]{SC}), the two-sided Gaussian estimates for $p_t^h(x,y)$ are equivalent to doubling together with the scaled $L^2$ Poincar\'{e} inequalities, both for the measure $h^2\mu$. But since $h\sim 1$, these are a direct consequence of the corresponding inequalities \eqref{P} and \eqref{D} for the measure $\mu$. This concludes the proof of Theorem \ref{heat_kernel}.

\cqfd
Upper-estimates for the heat kernel of Schr\"{o}dinger operators have consequences for the boundedness of the Riesz transform, as we explain now. Ten years ago, it has been discovered by Coulhon and Duong \cite{CDu2} (see also Sikora \cite{Sik}) that Gaussian estimates for the heat kernel of the Hodge-De Rham Laplacian $\vec{\Delta}=dd^\star+d^\star d$ acting on $1-$forms has consequences for the boundedness on $L^p$, $p\in (2,\infty)$ of the Riesz transform $d\Delta^{-1/2}$. More precisely, they show that if $M$ satisfies \eqref{D} and \eqref{UE}, and if the heat kernel of $\vec{\Delta}$ acting on $1$-forms has Gaussian estimates, then the Riesz transform $d\Delta^{-1/2}$ is bounded on $L^p$, for all $p\in (1,\infty)$. The Bochner formula asserts that $\vec{\Delta}$ acting on $1-$forms can be written as

$$\vec{\Delta}=-\nabla^*\nabla +\mathrm{Ric},$$
where for every $x\in M$, $\mathrm{Ric}_x$ is an symmetric endomorphism of the fiber $T^*_x M$, canonically associated with the Ricci curvature. Define a potential $\mathcal{V}$ by the requirement that $\mathcal{V}(x)$ is the lowest eigenvalue of $\mathrm{Ric}_x$. As a consequence of the domination theory, for every $t>0$ and $x,y\in M$,

$$||e^{-t\vec{\Delta}}(x,y)||\leq |e^{-t(\Delta+\mathcal{V})}(x,y)|.$$
Therefore, if the heat kernel of $\Delta+\mathcal{V}$ has Gaussian estimates, so does the heat kernel of $\vec{\Delta}$. This has been used in \cite{CZ} to prove that if $M$ satisfies \eqref{D} and \eqref{UE}, $\mathcal{V}_-$ is strongly positive with respect to $\Delta$ and, for $\delta>0$ small enough,

\begin{equation}\label{int_heat}
\sup_{x\in M}\int_0^\infty \int_M \frac{1}{V(x,\sqrt{t})}e^{-\frac{d^2(x,y)}{t}}\mathcal{V}_-(y)dydt<\delta,
\end{equation}
then $e^{-t\vec{\Delta}}$ has Gaussian estimates and the Riesz transform is bounded on $L^p$ for every $p\in (1,\infty)$. Notice that \eqref{int_heat} is closely related to the validity of

\begin{equation}\label{int_G}
\sup_{x\in M}\int_M G(x,y)\mathcal{V}_-(y)dy<\varepsilon,
\end{equation}
for $\varepsilon$ small enough (indeed, if $M$ satisfies the Poincar\'{e} inequalities \eqref{P}, then \eqref{int_heat} for small $\delta$ is equivalent to \eqref{int_G} for small $\varepsilon$). Under \eqref{P}, Takeda \cite[Theorem 2]{T1} proves a far better result: in order that $e^{-t\vec{\Delta}}$ has Gaussian estimates, it is enough for $\mathcal{V}_-$ to be strongly positive and in the Kato class at infinity $K^\infty(M)$, and for $\mathcal{V}_+$ to be $(H,\mathbf{1})$-bounded. That is, the smallness of $\mathcal{V}$ in an integral sense is required only at infinity, and not globally as in \eqref{int_heat} or \eqref{int_G}. As a consequence of Theorem \ref{heat_kernel}, we can get rid of the extra hypothesis \eqref{P}:

\begin{Cor}\label{Riesz_small}

Let $M$ be a non-parabolic Riemannian manifold, endowed with its Riemannian measure, satisfying \eqref{D} and \eqref{UE}. Let $\mathcal{V}(x)$ be the lowest eigenvalue of $\mathrm{Ric}_x$; assume that $\mathcal{V}_-$ satisfies the condition $(K^\infty,1)$, $\mathcal{V}_+$ is $(H,\mathbf{1})$-bounded, and $\Delta-\mathcal{V}_-$ is subcritical. Then the Riesz transform on $M$ is bounded on $L^p$ for every $p\in (1,\infty)$.

\end{Cor}
This greatly improves on \cite[Corollary 3.1]{CZ}. In light of \cite[Theorem 4]{D2}, it is natural to make the following conjecture:

\begin{Conj}\label{conj_riesz}

Let $M$ be a non-parabolic Riemannian manifold, endowed with its Riemannian measure, satisfying \eqref{D} and \eqref{UE} such that $\vec{\Delta}$ is strongly positive, and $|\mathrm{Ric}_-|$ satisfies $(K^\infty,1)$. Then the heat kernel of the Hodge Laplacian has Gaussian estimates, and the Riesz transform is bounded on $L^p$ for every $p\in (1,\infty)$.

\end{Conj}

\begin{Rem}
{\em
As this article was being written, Conjecture \ref{conj_riesz} has been proved in the forthcoming work \cite{CDS}.
}
\end{Rem}

\section{A criterion for $p$-hyperbolicity}

As we have already mentioned in the Preliminaries, the $p$-hyperbolicity of a manifold is tightly related to its volume growth. The aim of this section is to prove the following characterization of $p$-hyperbolicity in term of volume growth:

\begin{Thm}\label{car_hyp}

Let $M$ be a manifold satisfying \eqref{UE} and \eqref{D}, and let $p_0\in (2,\infty]$. The following are equivalent:

\begin{enumerate}

\item[(i)] For all $p\in (1,p_0)$, $M$ is $p$-hyperbolic.

\item[(ii)] For all $p\in (1,p_0)$, and for some (all) point $x_0\in M$, there is a constant $C=C(x_0,p)$ such that for all $t\geq 1$,

$$V(x_0,t)\geq Ct^p.$$

\end{enumerate}

\end{Thm}

\begin{Rem}
{\em 
Notice that in contrast to \cite{H} or \cite[Proposition 3.4]{CHSC}, no {\em global} Poincar\'{e}-type inequality is required. Furthermore, the volume growth condition is particularly simple. The drawback is that the above equivalence may not hold at the boundary of the interval, i.e. for $p_0$ itself. However, in many examples, the set of $p$'s such that $M$ is $p$-hyperbolic is an {\em open} set, so it does not seem to be a too serious restriction.
}
\end{Rem}

\begin{Cor}

Let $M$ be a manifold satisfying \eqref{UE} and \eqref{D}, and $(P)_{loc}$, and assume that $\kappa$, the parabolic dimension of $M$, satisfies $\kappa>2$. Then $\kappa$ is the supremum of $p$'s having the following property: for some (all) $x_0$ in $M$, there is a constant $C=C(p,x_0)$ such that, for all $t\geq1$,

$$V(x_0,t)\geq C(p,x_0)t^p.$$

\end{Cor}
Before being able to prove Theorem \ref{car_hyp}, we need to some preliminary results. For $p\in (1,\infty)$, let us introduce the following volume condition: for some (all) $x\in M$,

\begin{equation}\label{vol_p}\tag{$\tilde{V}_p$}
\int_1^\infty \frac{dt}{V(x,t)^{1/p}}<\infty.
\end{equation}

\begin{Lem}\label{Delta_hyp}

Assume that \eqref{vol_p} holds. Then, $\Delta^{-1/2}:L^p\to L^p_{loc}$ is bounded.

\end{Lem}

\begin{proof}

Let us start by some preliminary observations. Write

$$\Delta^{-1/2}=\int_0^1 e^{-t\Delta}\frac{dt}{\sqrt{t}}+\int_1^\infty e^{-t\Delta}\frac{dt}{\sqrt{t}}=R+S.$$
Using the fact that $e^{-t\Delta}$ is a contraction semi-group on $L^p$, $p\in [1,\infty]$, we see that $R$ is bounded on $L^p$, $p\in [1,\infty]$. Hence, $\Delta^{-1/2}:L^p\to L^p_{loc}$ if and only if $S : L^p\to L^p_{loc}$. We are going to see that $S$ is in fact bounded from $L^p$ to $L^\infty_{loc}$. Introduce the notation: for $x\in M$,
%Furthermore, for $k\geq1$,

%$$\Delta^kS=\int_1^\infty \Delta^ke^{-t\Delta}\frac{dt}{\sqrt{t}},$$
%and by analyticity of the heat semi-group, for every $p\in (1,\infty)$, $||\Delta^ke^{-t\Delta}||_{p,p}\leq Ct^{-k}$, so that $\Delta^kS$ is bounded on $L^p$. As a consequence, by elliptic regularity (see \cite[]{GT}), if we assume that $S : L^p\to L^p_{loc}$, then $S : L^p\to L^\infty_{loc}$. Consequently,  $\Delta^{-1/2}:L^p\to L^p_{loc}$ if and only if $S : L^p\to L^\infty_{loc}$.\\

$$L^\infty(x):=L^\infty(\{x\}).$$
We estimate

$$||S||_{L^p,L^\infty(x)}\leq \int_1^\infty ||e^{-t\Delta}||_{L^p,L^\infty(x)}\frac{dt}{\sqrt{t}}.$$
From the Gaussian estimate satisfied by $e^{-t\Delta}$,

$$||e^{-t\Delta}||_{L^p,L^\infty(x)}\leq \frac{C}{V(x,\sqrt{t})^{1/p}}$$
(this follows by interpolation from the obvious cases $p=1,\infty$). Therefore,

$$||S||_{L^p,L^\infty(x)}\leq C \int_1^\infty \frac{dt}{\sqrt{t}V(x,\sqrt{t})^{1/p}}=C\int_1^\infty \frac{dt}{V(x,t)^{1/p}
}.$$
By hypothesis, the integral $\int_1^\infty \frac{dt}{V(x,t)^{1/p}
}$ converges for some (all) $x\in M$. It is easy to see that the convergence is uniform with respect to $x$, if $x$ belongs to a fixed compact set. Consequently, $S:L^p\to L^\infty_{loc}$, and therefore 

$$\Delta^{-1/2}:L^p\to L^p_{loc}.$$

\end{proof}
There is also a link between $p$-hyperbolicity and the fact that $\Delta^{-1/2}$ is bounded from $L^p$ to $L^p_{loc}$. Indeed, let us recall the following result from \cite[Proposition 2.2]{D2}:

\begin{Pro}\label{D_p_hyp}

Let $p\in (1,\infty)$ such that $M$ is $p$-hyperbolic. Assume that the Riesz transform $d\Delta^{-1/2}$ is bounded on $L^p$. Then 

$$\Delta^{-1/2}:L^p\to L^p_{loc}$$
is a bounded operator. Conversely, if the Riesz transform is bounded on $L^{p'}$, $\frac{1}{p}+\frac{1}{p'}=1$, and if

$$\Delta^{-1/2}:L^p\to L^p_{loc}$$
is a bounded operator, then $M$ is $p$-hyperbolic.

\end{Pro}
\noindent\textit{Proof of Theorem \ref{car_hyp}:} consider the following four assertions:

\begin{enumerate}

\item For all $p\in (1,p_0)$, $M$ is $p$-hyperbolic.

\item For all $p\in (1,p_0)$, \eqref{Vol_p} is satisfied.

\item For all $p\in (1,p_0)$, \eqref{vol_p} is satisfied.

\item For all $p\in (1,p_0)$, $\Delta^{-1/2}:L^p\to L^p_{loc}.$

\end{enumerate}
We are going to show the chain of implications 

$$(1)\Rightarrow(2)\Rightarrow(3)\Rightarrow(4)\Rightarrow(1),$$
and that (ii) is equivalent to (3), which will prove the theorem. First, by \cite[Corollary 3.2]{CHSC}, (1) implies (2).

Let us assume (2), and let $p\in (1,p_0)$. By H\"{o}lder's inequality,

$$\int_1^\infty \frac{dt}{V(x,t)^{1/p}}\leq \left(\int_1^\infty \left(\frac{t}{V(x,t)}\right)^{\frac{q}{p(q-1)}}\right)^{\frac{q-1}{q}}\left( \int_1^\infty t^{-\frac{q}{p}}dt\right)^{\frac{1}{q}}.$$
If $p<q<p_0$, then the second integral converges. For the first one, define $r$ by

$$\frac{q}{p(q-1)}=\frac{1}{r-1}$$
If $q$ is close enough to $p$, then $p<r<p_0$. By (2), ($V_r$) is satisfied, that is 

$$\int_1^\infty \left(\frac{t}{V(x,t)}\right)^{\frac{1}{r-1}}du<\infty.$$
Therefore, the integral $\int_1^\infty \frac{dt}{V(x,t)^{1/p}}$ converges. This shows (2)$\Rightarrow$(3). 

The implication (3)$\Rightarrow$(4) is the result of Lemma \ref{Delta_hyp}. 

Assume (4). Since $M$ satisfies \eqref{UE} and \eqref{D}, by \cite{CDu} the Riesz transform $d\Delta^{-1/2}$ is bounded on $L^s$, $s\in (1,2]$. Hence, by Proposition \ref{D_p_hyp}, $M$ is $p$-hyperbolic for all $p\in (2,p_0)$. This shows (4)$\Rightarrow$ (1).

 It remains to show that (ii) is equivalent to (3). Obviously, (ii) implies (3). The converse is elementary, and has been  proved in \cite{CMO}. For the sake of completeness, we reproduce the argument here. Let us assume that \eqref{vol_p} holds. Let us denote $f(t)=V(x_0,t)^{-1/p}$, then $f$ is non-negative, non-increasing and integrable over $(1,\infty)$. It follows that for every $t>1$,

$$(t-1)f(t)\leq \int_1^t f(u)\,du \leq \int_1^\infty f(u)\,du=C<\infty.$$
Therefore,

$$f(t)\leq C(t-1)^{-1},$$
which implies that for every $t\geq 2$,

$$V(x_0,t)\geq C_pt^p.$$
This shows that (3) implies (ii), and concludes the proof of the theorem.

\cqfd
In fact, using the same ideas as in the proof of Theorem \ref{car_hyp}, one can also treat {\em relative} volume estimates:

\begin{Thm}\label{capacity}

Let $M$ satisfying \eqref{D} and \eqref{UE}. Let $p_0\in (2,\infty]$, $r_0>0$ and $x_0\in M$. Consider the following two inequalities, where $C_p$ is a positive constant:

\begin{enumerate}

\item[(i)] for all $p\in (1,p_0)$, $$\mathrm{Cap}_p(B(x,r))\geq C_p\frac{V(x,r)}{r^p}.$$

\item[(ii)] for all $p\in (1,p_0)$, $$\frac{V(x,t)}{V(x,r)}\geq C_p\left(\frac{t}{r}\right)^{p},\qquad\forall t\geq r.$$

\end{enumerate}
Then (i) and (ii), for all $r\geq r_0$ and $x=x_0$, are equivalent.

\end{Thm} 

\begin{Rem}
{\em
One can also show the equivalence of (i) and (ii) in Theorem \ref{capacity} under one of the following alternative conditions on $x$ and $r$:

\begin{itemize}

\item For all $x\in M$, and $r=r_0$.

\item For all $x\in M$ and all $r\geq r_0$.

\end{itemize}
The proof is the same.

}
\end{Rem}

\begin{Rem}
{\em
The result of Theorem \ref{capacity} shows that in \cite[Theorem A]{Ca}, the second part does not improve on the first one. 
}
\end{Rem}

\begin{Def}
{\em 
Let $p\in (1,\infty)$. A manifold $M$ such that, for some positive constant $C_p$, for all $x\in M$ and all $r>0$, 

$$\mathrm{Cap}_p(B(x,r))\geq C_p\frac{V(x,r)}{r^p},$$
will be called $p$-{\em regular}.

}

\end{Def}

\begin{proof}

Let us introduce a third inequality:

\begin{itemize}
\item[(iii)] $$\int_{r}^\infty\left(\frac{V(x,r)}{V(x,t)}\right)^{1/p}\,dt\leq C_p\,r.$$
\end{itemize}

We show the equivalence of (i), (ii) and (iii) under the condition (b), the two other cases being similar. We first show that (i) and (iii) are equivalent. According to \cite{CHSC}, the $p$-capacity of $\bar{B}(x_0,r)$ can be estimated by

$$\mathrm{Cap}_p(\bar{B}(x_0,r))\leq \left(\int_r^\infty \left(\frac{t}{V(x_0,t)}\right)^{\frac{1}{p-1}}\,dt\right)^{1-p}.$$
Assume that (i) holds true for all $1<p<p_0$ and all $r\geq1$, at the point $x=x_0$. One obtains that for all $r\geq1$,

$$\int_r^\infty \left(\frac{tV(x_0,r)}{r^pV(x_0,t)}\right)^{\frac{1}{p-1}}\,dt\leq C_p^{1-p}.$$
The argument based on H\"older's inequality and used in the proof of Theorem \ref{car_hyp} (see the implication (2)$\Rightarrow$(3)) shows that for every $p\in (1,p_0)$, (iii) is satisfied for all $r\geq r_0$, at the point $x_0$. 

Conversely, assume that (iii) is satisfied for all $r\geq r_0$, at the point $x_0$. As in the proof of Proposition \ref{Delta_hyp}, let us introduce the operators:

$$R_r=\int_0^{r^2}e^{-t\Delta}\frac{dt}{\sqrt{t}},$$
and

$$S_r=\int_{r^2}^\infty e^{-t\Delta}\frac{dt}{\sqrt{t}},$$
so that 

$$\Delta^{-1/2}=R_r+S_r.$$
Using the uniform boundedness of $e^{-t\Delta}$ on $L^p$, we see that

$$||R_r||_{p,p}\leq Cr.$$
We want to estimate $||S_r||_{L^p\to L^p(B(x_0,r))}$. The computations done in the proof of Proposition \ref{Delta_hyp} show that

$$||S_r||_{L^p\to L^\infty(x)}\leq C\int_r^\infty \frac{dt}{V(x,t)^{1/p}}.$$
Using \eqref{D}, we see that for every $x\in B(x_0,r)$,

$$\begin{array}{rcl}
||S_r||_{L^p\to L^\infty(x)}&\leq& \frac{C}{V(x_0,r)^{1/p}}\int_r^\infty \left(\frac{V(x_0,r)}{V(x,t)}\right)^{1/p}\,dt\\\\
&\leq & \frac{C}{V(x_0,r)^{1/p}}\int_r^\infty \left(\frac{V(x_0,r)}{V(x_0,t)}\right)^{1/p}\,dt.
\end{array}$$ 
Thus,

$$\begin{array}{rcl}
||S_r||_{L^p\to L^p(B(x_0,r))}&\leq& V(x_0,r)^{1/p}||S_r||_{L^p\to L^\infty(B(x_0,r))}\\\\
&\leq & \int_r^\infty \left( \frac{V(x_0,r)}{V(x_0,t)}\right)^{1/p}\,dt,
\end{array}$$ 
and by (iii) one obtains

$$||S_r||_{L^p\to L^p(B(x_0,r))}\leq Cr.$$
Consequently,

$$||\Delta^{-1/2}||_{L^p\to L^p(B(x_0,r))}\leq Cr.$$
We now recall an argument from \cite[Proposition 2.2]{D2} (we refer to this paper for more details): from the above inequality, one can conclude that

$$||u||_{L^p(B(x_0,r))}\leq Cr||\Delta^{1/2}u||_p,\qquad\forall u\in C_0^\infty(M).$$
Since $M$ satisfies \eqref{UE} and \eqref{D}, by \cite{CDu} the Riesz transform $d\Delta^{-1/2}$ is bounded on $L^s$, $s\in (1,2]$. In particular, it is bounded on $L^{p'}$, $\frac{1}{p}+\frac{1}{p'}=1$. But it is well-known (see \cite[Proposition 2.2]{CDu2}) that the boundedness of the Riesz transform on $L^{p'}$, implies the inequality

$$||\Delta^{1/2}u||_p\leq C||\nabla u||_p,\qquad\forall u\in C_0^\infty(M).$$
Thus,

$$||u||_{L^p(B(x_0,r))}\leq Cr||\nabla u||_p,\qquad\forall u\in C_0^\infty(M).$$
Using this inequality for a sequence $(u_n)_{n\in \mathbb{N}}$ of $C_0^\infty$ functions, such that $u_n\geq1$ on $B(x_0,r)$ and 

$$\lim_{n\to\infty}\int_M|\nabla u_n|^p=\mathrm{Cap}_p(\bar{B}(x_0,r)),$$
one finds that

$$V(x_0,r)^{1/p}\leq Cr\left(\mathrm{Cap}_p(B(x_0,r))\right)^{1/p},$$
that is (i). This concludes the proof of the equivalence between (i) and (iii). We now prove the equivalence between (ii) and (iii). Of course, (ii) implies (iii). The converse uses the same idea as in the proof of Theorem \ref{car_hyp}: for $r\geq r_0$, denote $f_r(t)=\left(\frac{V(x_0,r)}{V(x_0,t)}\right)^{1/p}$, then $f_r$ is non-negative, non-increasing, so that by (ii),

$$(t-r)f_r(t)\leq \int_r^t f_r(s)\,ds\leq \int_r^\infty f_r(s)\,ds\leq Cr.$$
Therefore,

$$f_r(t)\leq Cr(t-r)^{-1},\qquad \forall t\geq r,$$
that is,

$$\left(\frac{V(x_0,r)}{V(x_0,t)}\right)\geq C_p\left(\frac{t-r}{r}\right)^p,\qquad \forall t\geq r,$$
which by \eqref{D} implies (ii).

\end{proof}

\section{Weighted spaces}

In general, a manifold $M$ for which \eqref{D} and \eqref{UE} hold, do not need to satisfy either the Sobolev inequality \eqref{S}, or the non-collapsing of balls \eqref{NC}. In particular, in general the operators $\Delta^{-\alpha}$ do not behave well with respect to the $L^p$ spaces. In order to overcome these difficulties, weighted spaces have to be considered. Weighted estimates on the $L^p$ spaces for the heat kernel have recently been considered in \cite{BCS} (see also \cite{AO}), and their equivalence with weighted resolvent estimates on $L^p$ spaces has been demonstrated. Here, we will push this idea one step further, and introduce a natural class of weighted $L^p$ spaces. For $p\in [1,\infty]$, define

$$L^p_V(M):= L^p\left(M,\frac{d\mu(x)}{V(x,1)}\right).$$
Notice that $L^\infty_V(M)=L^\infty(M)$. Note also that by Bishop-Gromov, if the Ricci curvature of $M$ is bounded from below, and $M$ satisfies the non-collapsing \eqref{NC}, then  $V(x,1)\sim 1$ and $L^p_V(M)$ identifies to $L^p(M).$ Let us recall the following result from \cite{AO} or \cite{BCS}: 

\begin{Pro}\label{vEv}

Let $1\leq p\leq q\leq\infty$. Let $\delta$ and $\gamma$ be real numbers so that $\delta+\gamma=\frac{1}{p}-\frac{1}{q}$. Let $P$ be a non-negative, self-adjoint operator on $L^2(M,\mu)$, such that the Gaussian estimates hold for its heat kernel. Then 

$$\sup_{t>0}||V(\cdot,\sqrt{t})^{\gamma}e^{-tP}V(\cdot,\sqrt{t})^{\delta}||_{p,q}<\infty.$$

\end{Pro}
We will need a similar estimate for the gradient of the heat kernel:

\begin{Pro}\label{gvEv}

Assume \eqref{D} and \eqref{UE}, and that for some $q\in (1,\infty)$,

$$\sup_{t>0}\sqrt{t}||\nabla e^{-t\Delta}||_{q,q}<\infty.$$
Then for every $1\leq p< q$, and $\delta$, $\gamma$ real numbers so that $\delta+\gamma=\frac{1}{p}-\frac{1}{q}$,

$$\sup_{t>0}\sqrt{t}||V(\cdot,\sqrt{t})^{\gamma}\nabla e^{-t\Delta}V(\cdot,\sqrt{t})^{\delta}||_{p,q}<\infty.$$

\end{Pro}

\begin{Rem}

{\em

By analyticity on $L^p$ of the heat semi-group of $\Delta$, if the Riesz transform is bounded on $L^q$, then the gradient estimate

$$\sup_{t>0}\sqrt{t}||\nabla e^{-t\Delta}||_{q,q}<\infty$$
holds. Conversely, it is shown in \cite{ACDH} that if the above gradient estimate holds, and $M$ satisfies \eqref{D} and \eqref{P}, then the Riesz transform is bounded on $L^p$, for all $p\in (1,q)$.

}
\end{Rem}

\begin{proof}

Denote $V_{\sqrt{t}}(x):=V(x,\sqrt{t}).$ Writing $\nabla e^{-t\Delta}=\nabla e^{-\frac{t}{2}\Delta}e^{-\frac{t}{2}\Delta}$ and using the result of Proposition \ref{vEv} and \eqref{D}, it is easy to see that the result holds true if $\gamma=0$. In order to prove the result for all $\gamma$, we will use some ideas from \cite{BCS}. By complex interpolation for the family of operators

$$T_z=V(\cdot,\sqrt{t})^{\gamma_1z+(1-z)\gamma_2}\nabla e^{-t\Delta}V(\cdot,\sqrt{t})^{\delta_1z+(1-z)\delta_2},$$
and using $\sup_{t>0}\sqrt{t}||\nabla e^{-t\Delta}||_{q,q}<\infty$, it is enough to prove the result for $p=2$. Therefore, let us take $\delta$, $\gamma$ real numbers so that $\delta+\gamma=\frac{1}{2}-\frac{1}{q}$. Define $\Phi=F_a$ to be the Fourier transform of $t\mapsto (1-t^2)_+^a$. It can be checked (see the proof of Proposition 4.1.6. in \cite{BCS}) that the following transmutation formula holds:

$$e^{-t\Delta}=\int_0^\infty F_a(\sqrt{st\Delta})s^{a+\frac{1}{2}}e^{-s/4}\,ds.$$
Therefore, using \eqref{exp_D},

$$\begin{array}{lcr}
\sqrt{t}||V_{\sqrt{t}}^{\gamma}\nabla e^{-t\Delta}V_{\sqrt{t}}^{\delta}||_{2,q}\leq  \\\\
\hskip8mm\int_0^\infty \sqrt{st}||V_{\sqrt{st}}^{\gamma}\nabla F_a(\sqrt{st\Delta}) V_{\sqrt{st}}^{\delta}||_{2,q}\left(\sqrt{s}+\frac{1}{\sqrt{s}}\right)^{\nu(|\gamma|+|\delta|)}s^{a}e^{-s/4}\,ds.
\end{array}$$
Hence, if $a$ is big enough,

$$\sup_{t>0}\sqrt{t}||V_{\sqrt{t}}^{\gamma}\nabla e^{-t\Delta}V_{\sqrt{t}}^{\delta}||_{2,q}\leq \sup_{t>0} \sqrt{t}||V_{\sqrt{t}}^{\gamma}\nabla F_a(\sqrt{t\Delta}) V_{\sqrt{t}}^{\delta}||_{2,q}.$$
Setting $T_r=\nabla F_a(r\sqrt{\Delta})$, it follows from the finite propagation speed property for the wave equation that if $f_1$ has support in $B_1$, $f_2$ has support with $B_2$ and $d(B_1,B_2)>r$, then $\langle T_rf_1,f_2\rangle=0$. By \cite[Proposition 4.1.1.]{BCS},

$$||V_{r}^{\gamma} T_r V_{r}^{\delta}||_{2,q}\leq C||T_rV_r^{\frac{1}{2}-\frac{1}{q}}||_{2,q}.$$
By a straightforward adaptation of \cite[Lemma 4.1.4.]{BCS}, if $\sup_{\lambda}|(1+\lambda^2)^{N+1}\Phi(\lambda)|<\infty$, then 

$$||\nabla \Phi(\sqrt{t\Delta})V_{\sqrt{t}}^{\frac{1}{2}-\frac{1}{q}}||_{2,q}\leq C ||\nabla (I+t\Delta)^{-N}V_{\sqrt{t}}^{\frac{1}{2}-\frac{1}{q}}||_{2,q}.$$
One easily checks that the condition $\sup_{\lambda}|(1+\lambda^2)^{N+1}\Phi(\lambda)|<\infty$ is satisfied for $\Phi=F_a$ if $2N+1\leq a$. Therefore, if $a$ is large enough and $a\geq 2N+1$,

$$\sup_{t>0}\sqrt{t}||V_{\sqrt{t}}^{\gamma}\nabla e^{-t\Delta}V_{\sqrt{t}}^{\delta}||_{2,q}\leq C \sup_{t>0}\sqrt{t}||\nabla (I+t\Delta)^{-N}V_{\sqrt{t}}^{\frac{1}{2}-\frac{1}{q}}||_{2,q}.$$
We now use 

$$(I+t\Delta)^{-N}=\frac{1}{\Gamma(N)}\int_0^\infty e^{-s}s^{N-1}e^{-st\Delta}\,ds,$$
so that, using \eqref{exp_D},

$$\begin{array}{lcr}
\sup_{t>0}\sqrt{t}||\nabla (I+t\Delta)^{-N}V_{\sqrt{t}}^{\frac{1}{2}-\frac{1}{q}}||_{2,q}\leq  \\\\
\hskip8mmC\int_0^\infty e^{-s}s^{N-\frac{3}{2}}\left(\sqrt{s}+\frac{1}{\sqrt{s}}\right)^{\nu(|\delta|+|\gamma|)}\sqrt{st}||\nabla e^{-st\Delta}V_{\sqrt{st}}^{\frac{1}{2}-\frac{1}{q}}||_{2,q}\,ds.
\end{array}$$
Hence, if $N$ is big enough,

$$\sup_{t>0}\sqrt{t}||\nabla (I+t\Delta)^{-N}V_{\sqrt{t}}^{\frac{1}{2}-\frac{1}{q}}||_{2,q}\leq C\sup_{t>0} \sqrt{t}||\nabla e^{-t\Delta}V_{\sqrt{t}}^{\frac{1}{2}-\frac{1}{q}}||_{2,q}.$$
Consequently, if one chooses $N$ and $a$ big enough satisfying $a\geq 2N+1$, one obtains

$$\sup_{t>0}\sqrt{t}||V_{\sqrt{t}}^{\gamma}\nabla e^{-t\Delta}V_{\sqrt{t}}^{\delta}||_{2,q}\leq \sup_{t>0} \sqrt{t}||\nabla e^{-t\Delta}V_{\sqrt{t}}^{\frac{1}{2}-\frac{1}{q}}||_{2,q}.$$
This last quantity is finite by the remark we made at beginning of the proof, and this concludes the proof of Proposition \ref{gvEv}.

\end{proof}
We now rephrase Proposition \ref{vEv} and Proposition \ref{gvEv} in term of action of the heat kernel on the weighted spaces $L^p_V$. Let $P=\Delta+\mathcal{V}$ be a Schr\"{o}dinger operator. Using \eqref{exp_D}, we see that if the heat kernel of $P$ has Gaussian estimates, then

$$\begin{array}{rcl}
&&||V(\cdot,\sqrt{t})^{-1/q}e^{-tP}V(\cdot,\sqrt{t})^{1/p}||_{p,q}\\\\
&=&||\left(\frac{V(\cdot,\sqrt{t})}{V(\cdot,1)}\right)^{-1/q}V(\cdot,1)^{-1/q}e^{-tP}V(\cdot,1)^{1/p}\left(\frac{V(\cdot,\sqrt{t})}{V(\cdot,1)}\right)^{1/p}||_{p,q}\\\\
&\geq& C\left(\varphi_{p,q}(t)\right)^{-1}||e^{-tP}||_{L^p_V,L^q_V},
\end{array}$$
where

$$\varphi_{p,q}(t)=\left\{\begin{array}{rcl}
t^{-\frac{\nu'}{2p}+\frac{\nu}{2q}},\,t\geq1\\
t^{-\frac{\nu}{2p}+\frac{\nu'}{2q}},\,t\leq1
\end{array}\right.$$
Thus, one can rephrase Proposition \ref{vEv} in the following way:
%Using these estimates for $q=p$ and for $q=\infty$ together with Proposition \ref{vEv}, and interpolating, one gets

\begin{Cor}\label{heat_weight}

Assume \eqref{D} and let $P=\Delta+\mathcal{V}$ be a Schr\"{o}dinger operator whose heat kernel have Gaussian estimates. Let $1\leq p\leq q\leq\infty$, then for every $t>0$,

$$||e^{-tP}||_{L^p_V,L^q_V}\lesssim \varphi_{p,q}(t).$$

\end{Cor}
Concerning Proposition \ref{gvEv}, one has the following:

\begin{Cor}\label{grad_weight}

Assume \eqref{D} and \eqref{UE}, and that for some $q\in (1,\infty)$,

$$\sup_{t>0}\sqrt{t}||\nabla e^{-t\Delta}||_{q,q}<\infty.$$
Then for every $1\leq p\leq q$, 

$$\sqrt{t}||\nabla e^{-t\Delta}||_{L^p_V,L^q_V}\lesssim \varphi_{p,q}(t).$$

\end{Cor}

In Section 7, we will need the following slight generalization of a perturbation result by T. Coulhon and N. Dungey \cite[Theorem 2.1]{CD}: 

\begin{Pro}\label{C-D}

Let $M$ satisfying \eqref{D} and \eqref{UE}. Let $H=-\div A\nabla$, $A$ symmetric, and $a=A-Id$, and assume that $a\in L_V^q\cap L^\infty$. Let $p_0>2$ such that $\nabla \Delta^{-1/2}$ and $\nabla (I+H)^{-1/2}$ are bounded on $L^p$ for every $p\in (2,p_0)$, and such that $\nabla H^{-1/2}$ is bounded on $L^p$, for $p\in (p_0',2)$. Then $\nabla H^{-1/2}$ is bounded on $L^p$ for every $p\in (2,p_0)$.

\end{Pro}
T. Coulhon and N. Dungey show this result under the stronger assumptions that $a\in L^q\cap L^\infty$ and that for every $t\geq1$, $||e^{-t\Delta}||_{1,\infty}\leq Ct^{-D/2}$. Such an ultracontractivity estimate for $e^{-t\Delta}$ is known to be equivalent to a Nash inequality at infinity (\cite{CSC}), however in this paper, we want to work in greater generality and avoid this type of hypothesis. Therefore, we have to work with the weighted spaces $L^p_V$.

\begin{proof}

The only part of the proof of \cite[Theorem 2.1]{CD} where the ultracontractivity estimate $||e^{-t\Delta}||_{1,\infty}\leq Ct^{-D/2}$ and the hypothesis $a\in L^q$ are used, is to show that for $p\in (2,p_0)$, there is $\varepsilon>0$  such that

$$||a\nabla (I+t\Delta)^{-1}||_{p,p}\leq Ct^{-\frac{1}{2}-\varepsilon},\qquad\forall t>0.$$
Let us show how to prove a similar estimate in our context, which allows us to make Coulhon and Dungey's proof work. Write

$$(I+t\Delta)^{-1}=\int_0^\infty e^{-s}e^{-st\Delta}\,ds,$$
so that

$$||a\nabla (I+t\Delta)^{-1}||_{p,p}\leq \int_0^\infty e^{-s}||a\nabla e^{-st\Delta}||_{p,p}\,ds.$$
Using H\"{o}lder's inequality,

$$||a\nabla e^{-st\Delta}||_{p,p}\leq ||aV(\cdot,1)^{-1/q}||_q||V(\cdot,1)^{1/q}\nabla e^{-st\Delta}||_{p,r},$$
where $r$ is defined by $\frac{1}{p}=\frac{1}{q}+\frac{1}{r}$. Notice that 

$$||aV(\cdot,1)^{-1/q}||_q=||a||_{L^q_V}<\infty.$$
Since $a\in L^q_V\cap L^\infty$, by interpolation we can assume that $q$ is large enough so that $2<r<p_0$. Then, the Riesz transform $\nabla \Delta^{-1/2}$ is bounded on $L^r$, and by a classical argument involving analyticity of the heat semi-group $e^{-t\Delta}$ on $L^p$,

$$\sup_{t>0}\sqrt{t}||\nabla e^{-t\Delta}||_{r,r}<\infty.$$
By \eqref{exp_D} and Proposition \ref{gvEv},

$$\begin{array}{rcl}
||V(\cdot,1)^{1/q}\nabla e^{-st\Delta}||_{p,r}&=&||\left(\frac{V(\cdot,\sqrt{st})}{V(\cdot,1)}\right)^{-1/q}V(\cdot,\sqrt{st})^{1/q}\nabla e^{-st\Delta}||_{p,r}\\\\
&\leq& C(st)^{-\frac{\nu'}{2q}-\frac{1}{2}}+C(st)^{-\frac{\nu}{2q}-\frac{1}{2}}\\\\
&\leq& C(st)^{-\frac{1}{2}-\varepsilon_1}+C(st)^{-\frac{1}{2}-\varepsilon_2},
\end{array}$$
where $\varepsilon_1=\frac{\nu'}{2q}$, $\varepsilon_2=\frac{\nu}{2q}$. Therefore, if $q$ is big enough so that $\varepsilon_1<\frac{1}{2}$, $\varepsilon_2<\frac{1}{2}$,

$$||a\nabla (I+t\Delta)^{-1}||_{p,p}\leq C(t^{-\frac{1}{2}-\varepsilon},\qquad\forall t>0,$$
with $\varepsilon=\min(\varepsilon_1,\varepsilon_2)$. With this at hand, the proof of Proposition \ref{C-D} follows the lines of the proof of \cite[Theorem 2.1]{CD}.

\end{proof}
Let us conclude this section by presenting a sufficient condition for a potential to belong to the Kato class at infinity, which generalizes Example \ref{Sob_Kato} to manifolds that do not satisfy \eqref{S}:

\begin{Pro}\label{cond_Kato}

Let $M$ satisfying \eqref{D} and \eqref{UE}, and let $\mathcal{V}\in L_V^{\frac{\nu'}{2}-\varepsilon}\cap L_V^{\frac{\nu}{2}+\varepsilon}$, for some $\varepsilon>0$. Then $\mathcal{V}\in K^\infty(M)$, and there is $q\in [1,\infty)$ such that $\Delta^{-1}|\mathcal{V}|\in L^q_V$.

\end{Pro}

\begin{proof}

Let us show that for some constant $C$ independent of $\mathcal{V}$,

\begin{equation}\label{inf_est}
||\Delta^{-1}|\mathcal{V}|||_\infty\leq C(||\mathcal{V}||_{ L_V^{\frac{\nu'}{2}-\varepsilon}}+||\mathcal{V}||_{L_V^{\frac{\nu}{2}+\varepsilon}}).
\end{equation}
Write

$$\Delta^{-1}=\int_0^\infty e^{-t\Delta}\,dt,$$
so that

$$||\Delta^{-1}|\mathcal{V}|||_\infty\leq \int_0^\infty ||e^{-t\Delta}|\mathcal{V}|||_{\infty}\,dt.$$
Using Corollary \ref{heat_weight} with $P=\Delta$, we get

$$\begin{array}{lll}
||\Delta^{-1}|\mathcal{V}|||_\infty&\leq& \left(\int_0^1 t^{-\frac{\nu}{\nu+2\varepsilon}}\,dt\right)||\mathcal{V}||_{L_V^{\frac{\nu}{2}+\varepsilon}}+ \left(\int_1^\infty t^{-\frac{\nu'}{\nu'-2\varepsilon}}\,dt\right)||\mathcal{V}||_{L_V^{\frac{\nu'}{2}-\varepsilon}}\\\\
&\leq&  C(||\mathcal{V}||_{L_V^{ \frac{\nu'}{2}-\varepsilon}}+||\mathcal{V}||_{L_V^{\frac{\nu}{2}+\varepsilon}}),
\end{array}$$
which shows \eqref{inf_est}. For $k\geq 0$, let $\mathcal{V}_k=\mathcal{V}\chi_k$. Then 

$$\lim_{k\to \infty}\left(||\mathcal{V}_k||_{L_V^{ \frac{\nu'}{2}-\varepsilon}}+||\mathcal{V}_k||_{L_V^{\frac{\nu}{2}+\varepsilon}}\right)=0,$$
which implies, by \eqref{inf_est} applied to $\mathcal{V}_k$,

$$\lim_{k\to \infty}\sup_{x\in M}\int_{\Omega_k^*}G(x,y)|\mathcal{V}(y)|\,dy=0,$$
i.e. $\mathcal{V}\in K^\infty(M)$. Let us now check that $\Delta^{-1}\mathcal{V}\in L^q_V$ for some $q<\infty$. We use again Corollary \ref{heat_weight}, to get

$$||\Delta^{-1}|\mathcal{V}|||_{L_V^q}\leq \left(\int_0^1 t^{-\frac{\nu}{\nu+2\varepsilon}+\frac{\nu}{2q}}\,dt\right)||\mathcal{V}||_{L_V^{\frac{\nu}{2}+\varepsilon}}+ \left(\int_1^\infty t^{-\frac{\nu'}{\nu'-2\varepsilon}+\frac{\nu'}{2q}}\,dt\right)||\mathcal{V}||_{L_V^{\frac{\nu'}{2}-\varepsilon}}.$$
Consequently, if $q$ is big enough so that

$$-\frac{\nu'}{\nu'-2\varepsilon}+\frac{\nu'}{2q}<-1,$$
then $||\Delta^{-1}|\mathcal{V}|||_{L_V^q}<\infty.$

\end{proof}

\section{Riesz transform with potential}

{\em From now on, $\mu$ is assumed to be the Riemannian measure on $M$.}\\

\noindent In this section, we will obtain boundedness and unboundedness results for the Riesz transform with potential $d(\Delta+\mathcal{V})^{-1/2}$. Let us start by introducing a definition: we say that $V$ satisfies condition $(L_p)$ if there is a constant $C$ such that for every $x\in M$,

\begin{equation}\label{L_p}\tag{$L_p$}
||\mathcal{V}||_{L^p(B(x,1))}\lesssim V(x,1)^{1/p'},
\end{equation}
where $\frac{1}{p}+\frac{1}{p'}=1$. Obviously, the validity of \eqref{L_p} implies that $V\in L_{loc}^p$. Also, if $M$ has Ricci curvature bounded from below and \eqref{NC} holds, then the condition \eqref{L_p} for $p=\infty$, is equivalent to $\mathcal{V}\in L^\infty$. Let us now state our main technical result, from which will be derived various consequences:

\begin{Thm}\label{main1}

Let $M$ be a $2$-hyperbolic manifold, satisfying \eqref{D} and \eqref{UE}. Assume that for some $p_0>2$, the Riesz transform $d\Delta^{-1/2}$ on $M$ is bounded on $L^p$, for all $p\in(1,p_0)$. Recall the exponents $\nu,\nu'$ from \eqref{exp_D}. Let $\mathcal{V}$ be a subcritical potential, and consider the following two assumptions:

\begin{enumerate}

\item[(i)] $p_0\leq \nu$, $\nu'>2$, the Ricci curvature on $M$ is bounded from below, and $\mathcal{V}\in L^{\frac{\nu'}{2}-\varepsilon}\left(M,\frac{d\mu(x)}{V(x,1)}\right)\cap L^{\frac{\nu}{2}+\varepsilon}\left(M,\frac{d\mu(x)}{V(x,1)}\right)$ satisfies \eqref{L_p} for some $p>N$, $N>2$ being the topological dimension of $M$.

\item[(ii)] $p_0>\nu$, and $\mathcal{V}\in L^{q_1}\left(M,\frac{d\mu(x)}{V(x,1)}\right)\cap L^{q_2}\left(M,\frac{d\mu(x)}{V(x,1)}\right)$, where 

$$q_1=\frac{\nu'p_0}{p_0+\nu}-\varepsilon,\,\,q_2=\max\left(\frac{p_0}{2},\frac{\nu p_0}{p_0+\nu'}+\varepsilon\right).$$

\end{enumerate}
Assume that either $\mathrm{(i)}$ or $\mathrm{(ii)}$ is satisfied. Let $h\sim 1$ be the solution of $(\Delta+\mathcal{V})u=0$, provided by Theorem \ref{pos_sol}. Then the operator

$$d(\Delta+\mathcal{V})^{-1/2}-(d\log h)(\Delta+\mathcal{V})^{-1/2}$$
is bounded on $L^p$, for every $p\in [2,p_0)$. Equivalently, for every $p\in [2,p_0)$, the following inequality holds:

\begin{equation}\label{imp_Riesz}
||du-(d\log h)u||_p\leq C||(\Delta+\mathcal{V})^{1/2}u||_p,\qquad\forall u\in C_0^\infty(M).
\end{equation}

\end{Thm}

%and such that for some $\alpha>2$, the following lower bound for the volume of balls with large radius holds:

%\begin{equation}\label{vol}\tag{V}
%V(x,r)\geq Cr^\alpha,\qquad\forall x\in M,\,\forall r\geq1.
%\end{equation}
%Assume that on $M$, the Riesz transform $d\Delta^{-1/2}$ is bounded on $L^p$ for every $p\in (2,p_0)$. Let $V\in L^\infty_{uloc}(M)$ be a potential bounded from below with $\mathcal{V}_-\in K^\infty(M)$ and $\mathcal{V}_+$ $(H,1)-$bounded, such that $\Delta+\mathcal{V}$ is subcritical, and such that for some $q\in [1,\infty)$, $\Delta^{-1}|V|\in L^q$. Let $h\sim 1$ be the positive solution of $(\Delta+\mathcal{V})u=0$ provided by Theorem \ref{pos_sol}. Then the operator

%$$d(\Delta+\mathcal{V})^{-1/2}-(d\log h)(\Delta+\mathcal{V})^{-1/2}$$
%is bounded on $L^p$, for every $p\in (2,p_0)$. Equivalently, for every $p\in (2,p_0)$, the following inequality holds:

%\begin{equation}\label{imp_Riesz}
%||du-(d\log h)u||_p\leq C||(\Delta+\mathcal{V})^{1/2}u||_p,\qquad\forall u\in C_0^\infty(M).
%\end{equation}

%\end{Thm}
One of the main features of Theorem \ref{main1} is to provide an alternative inequality \eqref{imp_Riesz} in the case where the Riesz transform $d(\Delta+\mathcal{V})^{-1/2}$ is unbounded on $L^p$. In turn, as we shall see, the validity of the inequality \eqref{imp_Riesz} gives a necessary criteria (the $p-$hyperbolicity of $M$) for the boundedness of $d(\Delta+\mathcal{V})^{-1/2}$ on $L^p$, which will turn out to be sufficient under further assumptions on $\mathcal{V}$. 

\begin{Cor}\label{main2}

Under the assumptions of Theorem \ref{main1} and if $\mathcal{V}\not\equiv0$, a necessary condition for $d(\Delta+\mathcal{V})^{-1/2}$ to bounded on $L^p$, for some $p\in (2,p_0)$ is that  $M$ is $p$-hyperbolic.

\end{Cor}

\begin{Rem}\label{CMO}
{\em

In \cite[Theorem 6.1]{CMO}, a related result to Corollary \ref{main2} is proved: more precisely, the authors prove that if $\mathcal{V}\geq0$ and there exists a positive, bounded function $h$ such that $(\Delta+\mathcal{V})h=0$, then a necessary condition for $d(\Delta+\mathcal{V})^{-1/2}$ to be bounded on $L^p$ is that $p\leq \nu$. Notice that unlike our result, the existence of $h>0$ bounded, solution of $(\Delta+\mathcal{V})h=0$ is {\em assumed}, and moreover the necessary condition is $p>\nu$, instead of $M$ being $p$-hyperbolic. Thus, even in the case of a smooth, compactly supported non-negative potential $\mathcal{V}$ in $\R^n$, it is not possible from their result to conclude that $d(\Delta+\mathcal{V})^{-1/2}$ cannot be bounded on $L^n$. Notice also that in some cases, our Corollary \ref{main2} yields a much sharper condition than $p>\nu$: indeed, there are examples for which $\kappa$, the parabolic dimension, is strictly less than $\nu$. For instance, consider the Taub-NUT metric $g$ on $\R^4$ (see \cite{Leb}): it is a complete metric with zero Ricci curvature such that

$$V(x,r)\sim r^3,\qquad,\forall x\in \R^4,\,\forall r\geq1.$$
In fact, at infinity the metric is asymptotic to the product metric on $\R^3\times S^1$. Thus, $(\R^4,g)$ has parabolic dimension $\kappa$ equals to $3$, and by Bakry's celebrated result \cite{Bak}, the Riesz transform on $(\R^4,g)$ is bounded on $L^p$ for all $p\in (1,\infty)$. Thus, our Corollary \ref{main2} applies, and gives the necessary condition $p<3$, for the boundedness of $d(\Delta+\mathcal{V})^{-1/2}$. On the other hand, looking at balls of small radius in \eqref{exp_D}, one has $4\leq \nu$, and in fact $4=\nu$ by Bishop-Gromov. Thus, the necessary condition provided by \cite[Theorem 6.1]{CMO} is only $p\leq 4$. 

}

\end{Rem}
Theorem \ref{main1} has the following consequence for manifolds satisfying the Sobolev inequality and having Euclidean volume growth:

\begin{Cor}\label{main3}

Let $M$ be a manifold satisfying the Sobolev inequality \eqref{S} for some $n>2$, and having Euclidean volume growth:

$$V(x,r)\sim r^n,$$
for all $x\in M$ and $r>0$. Let $\mathcal{V}\in L^{\frac{n}{2}\pm\varepsilon}$, $\varepsilon>0$ be subcritical. Assume that the Riesz transform $d\Delta^{-1/2}$ is bounded on $L^p$, $p\in (1,n+\varepsilon)$, $\varepsilon>0$. Then $d(\Delta+\mathcal{V})^{-1/2}$ is bounded on $L^p$ if and only if $p\in (1,n)$.

\end{Cor} 

\begin{Rem}
{\em
From Corollary \ref{main3}, one recovers the first half of a result by C. Guillarmou and A. Hassell \cite[Theorem 1.5]{GH}, which states that if $M$ is an asymptotically conic manifold, and $\mathcal{V}=O(r^{-3})$, then $d(\Delta+\mathcal{V})^{-1/2}$ is bounded on $L^p$ if and only if $p\in \left(1,n\right)$. In particular, we obtain an elementary (i.e. without using the $b-$calculus) proof of Guillarmou and Hassell's result. 
}
\end{Rem}
In \cite[Theorem 3.9]{AO}, a boundedness result for the Riesz transform is proved. Actually, looking closely at the proof and using our Proposition \ref{cond_Kato}, together with Theorems \ref{heat_kernel} and \ref{car_hyp}, one can improve it and show:

\begin{Thm}\label{AO}

Let $M$ satisfying \eqref{D}, \eqref{UE} and \eqref{P_loc}. Let $\kappa$ the parabolic dimension of $M$, and assume that $\kappa>2$. Assume also that the Riesz transform on $M$ is bounded on $L^p$, for all $p\in (1,p_0)$. Let $\mathcal{V}\in L^{\frac{\nu'}{2}-\varepsilon}\left(M,\frac{d\mu(x)}{V(x,1)}\right)\cap L^{\frac{\nu}{2}+\varepsilon}\left(M,\frac{d\mu(x)}{V(x,1)}\right)$ be subcritical such that for all $p<\kappa$, close enough to $\kappa$,

$$\int_1^\infty \left|\left|\frac{|\mathcal{V}|^{1/2}}{V(\cdot,t)^{1/p}}\right|\right|_p\,dt<\infty.$$
Then, $d(\Delta+\mathcal{V})^{-1/2}$ is bounded on $L^p$, for every $p\in (1,\min(\kappa,p_0))$.

\end{Thm}

\begin{Rem}
{\em 

\begin{enumerate}

\item Actually, as follows from Theorem \ref{car_hyp}, $p\leq \kappa$ is {\em necessary} for having

$$\int_1^\infty \left|\left|\frac{|\mathcal{V}|^{1/2}}{V(\cdot,t)^{1/p}}\right|\right|_p\,dt<\infty.$$

\item Assume that $M$ satisfies \eqref{P_loc}. If $\mathcal{V}\in L^\infty$ is subcritical with compact support, by Theorem \ref{car_hyp} one has, for all $p\in (1,\kappa)$,

$$\int_1^\infty \left|\left|\frac{|\mathcal{V}|^{1/2}}{V(\cdot,t)^{1/p}}\right|\right|_p\,dt<\infty.$$
Therefore, $d(\Delta+\mathcal{V})^{-1/2}$ is bounded on $L^p$ for all $p\in(1,\min(\kappa,p_0))$. This shows that Corollary \ref{main2} is {\em optimal}.
\end{enumerate}
}
\end{Rem}
One can in fact complement Theorem \ref{AO}, and extend Corollary \ref{main3} to more general manifolds; in the next theorem, which is one of the main results of this article, we prove a sharp boundedness result for the Riesz transform with potential:

\begin{Thm}\label{main4}

Let $M$ satisfying \eqref{D}, \eqref{UE} and \eqref{P_loc}. Recall the exponents $\nu,\nu'$ from \eqref{exp_D}. Let $\kappa$ be the parabolic dimension of $M$, assume that $\kappa>2$ and that $M$ is $\kappa$-parabolic. Assume also that for every $p<\kappa$, close enough to $\kappa$, $M$ is $p$-regular. Assume that the Riesz transform on $M$ is bounded on $L^{p_0}$ for some $p_0>\nu$. Let $\mathcal{V}\in L^{\frac{\nu'}{2}-\varepsilon}\left(M,\frac{d\mu(x)}{V(x,1)}\right)\cap L^{\frac{\nu}{2}+\varepsilon}\left(M,\frac{d\mu(x)}{V(x,1)}\right) $ be subcritical. Then $d(\Delta+\mathcal{V})^{-1/2}$ is bounded on $L^p$ if and only if $p\in (1,\kappa)$. 

\end{Thm}

\begin{Rem}
{\em

The fact that $p<\kappa$ is sufficient for the boundedness on $L^p$ of $d(\Delta+\mathcal{V})^{-1/2}$ follows from Theorem \ref{AO}, together with our Theorems \ref{heat_kernel} and Theorems \ref{car_hyp}. Therefore, in a sense it is a consequence of the results of \cite{AO}. The converse, however, is entirely new.
}

\end{Rem}

%Specializing even more the results of Theorem \ref{main1}, one obtains the following rather special  -- but useful -- result:

%\begin{Cor}\label{main3}

%Assume that $M$ is a manifold satisfying the Poincare inequalities and the Sobolev inequality of dimension $n>2$, having Ricci curvature bounded from below and such that the volume of balls of large radius is Euclidean of dimension $n$:

%\begin{equation}\label{VEuc}
%V(x,r)\simeq r^n,\qquad \forall r\geq1.
%\end{equation}
%Assume also that the Riesz transform on $M$ is bounded on $L^{n+\varepsilon}$. Let $V$ be a subcritical potential in $L^{\frac{n}{2}\pm\varepsilon}$. Then the Riesz transform with potential $d(\Delta+\mathcal{V})^{-1/2}$ is bounded on $L^p$ iff $p\in (1,n)$.

%\end{Cor}
%The boundedness of $d(\Delta+\mathcal{V})^{-1/2}$ on $L^p$ for $p\in (1,n)$ under the assumptions of \ref{main3} follows also from the results of \cite{AO}. Our proof has the advantage of being a straightforward consequence of Theorem \ref{main1}. Furthermore, we will see in Section 4 that one can use the results of \cite{AO} in order to obtain a similar caractarization of the range of boundedness of $d(\Delta+\mathcal{V})^{-1/2}$ for some potentials, on manifolds having uniform volume growth. 
The rest of this subsection is devoted to the proof of Theorem \ref{main1}, Corollaries \ref{main2} and \ref{main3}, and Theorem \ref{main4}. One of the main technical ingredients in the proof of Theorem \ref{main1} is a perturbation result by Coulhon and Dungey (Theorem 4.1 in \cite{CD}).  

\begin{proof}

Denote $P=\Delta+\mathcal{V}$. Let $D$ be the first-order differential operator:

$$Du=d(h^{-1}u),$$
where $h\sim 1$ is the positive solution of $Pu=0$ given by Theorem \ref{pos_sol}. Consider the operator $T_h$ which is multiplication by $h$, and the $h$-transform $P_h=T_h^{-1}PT_h$. By \eqref{Ph}, $P_h$ is the weighted Laplacian $\Delta_{h^2}$, self-adjoint on $L^2(M,h^2\mu)\simeq L^2(M,d\mu)$. It is clear by spectral theory that

$$P_h^{-1/2}=T_h^{-1}P^{-1/2}T_h.$$
Notice also that

$$T_h^{-1}DT_h=h^{-1}d.$$
Consequently,

$$T_h^{-1}DP^{-1/2}T_h=h^{-1}d\Delta_{h^2}^{-1/2}.$$
Notice that $T_h$ is an isometry from $L^p(M,d\mu)$ to $L^p(M,h^pd\mu)$. Given that $h\sim 1$, there is a natural identification $L^p(M,h^p dx)\simeq L^p(M,dx)$, and so $DP^{-1/2}$ is bounded on $L^p(M,dx)$ if and only if $d\Delta_{h^2}^{-1/2}$ is bounded on $L^p(M,dx)$. We claim that 

\begin{Lem}\label{dL}

The operator $d\Delta_{h^2}^{-1/2}$ is bounded on $L^p$, for every $p\in [2,p_0)$.

\end{Lem}
The proof of this claim relies on the above-mentioned perturbation result of Coulhon and Dungey \cite[Theorem 4.1]{CD}, and is postponed. Assuming the result of Lemma \ref{dL} for the moment, we obtain that $D(\Delta+\mathcal{V})^{-1/2}$ is bounded on $L^p$, $p\in (2,p_0)$. But

$$D(\Delta+\mathcal{V})^{-1/2}=d(h^{-1})(\Delta+\mathcal{V})^{-1/2}+h^{-1}d(\Delta+\mathcal{V})^{-1/2},$$
hence the operator 

$$d(\Delta+\mathcal{V})^{-1/2}-(d\log h)(\Delta+\mathcal{V})^{-1/2}$$
is bounded on $L^p$, $p\in (2,p_0)$. It remains to prove the inequality \eqref{imp_Riesz}. Let $v\in C_0^\infty(M)$. By the above, one has for every $u\in L^p$,

\begin{equation}\label{bound}
||d(\Delta+\mathcal{V})^{-1/2}u-(d\log h)(\Delta+\mathcal{V})^{-1/2}u||_p\leq C||u||_p.
\end{equation}
We want to apply inequality \eqref{bound} with the choice $u=(\Delta+\mathcal{V})^{1/2}v$. In order for this to be licit, one has to prove that $(\Delta+\mathcal{V})^{1/2}C_0^\infty\subset L^p$. Conjugating by $h$, one sees that this is equivalent to $\Delta_{h^2}^{1/2}C_0^\infty\subset L^p$. This has been proved in \cite[Lemma 2.2]{D2}. Hence, plugging in \eqref{bound} the function $u=(\Delta+\mathcal{V})^{1/2}v$, one finds \eqref{imp_Riesz}. This concludes the proof of Theorem \ref{main1}.

\end{proof}

\noindent \textit{Proof of Corollary \ref{main2}:} Let us prove the first part. Assume that $d(\Delta+\mathcal{V})^{-1/2}$ is bounded on $L^p$ for some $p\in(2,p_0)$. Let $D=d\circ h^{-1}$. We claim that for every $q\in [2,\infty)$,

\begin{equation}\label{rever}
||(\Delta+\mathcal{V})^{1/2}u||_q\leq C||Du||_q,\qquad\forall u\in C_0^\infty(M).
\end{equation}
To show this, let us start with the inequality

\begin{equation}\label{weight_rev}
||\Delta_{h^2}^{1/2}v||_q\leq C||dv||_q,\qquad\forall v\in C_0^\infty(M).
\end{equation}
Inequality \eqref{weight_rev} for every $q\in [2,\infty)$ follows from $h\sim 1$ and the fact that $d\Delta^{-1/2}_{h^2}$ is bounded on $L^q$, $q\in (1,2]$ (this will be proved in Lemma \ref{1<p<2}), together with a classical duality argument (see \cite[Proposition 2.1]{CDu2}). Now, by \eqref{Ph},

$$h^{-1}(\Delta+\mathcal{V})^{1/2}h=\Delta_{h^2}^{1/2},$$
and we use inequality \eqref{weight_rev} with $u=hv$, to obtain that for every $u\in C_0^\infty(M)$,

$$||h^{-1}(\Delta+\mathcal{V})^{1/2}u||_q\leq C||Du||_q.$$
Since $h\sim1$, we get \eqref{rever}. By Theorem \ref{main1}, the operator $(dh)(\Delta+\mathcal{V})^{-1/2}$ has to be bounded on $L^p$, therefore

$$||(dh)u||_p\leq C||(\Delta+\mathcal{V})^{1/2}u||_p,\qquad\forall u\in C_0^\infty(M),$$
which, together with \eqref{rever} for $q=p$, implies that

$$||(dh)u||_p\leq C||Du||_p,\qquad\forall u\in C_0^\infty(M).$$
If we let $v=h^{-1}u$, we obtain that

$$||(d\log h)v||_p\leq C||dv||_p,\qquad\forall u\in C_0^\infty(M).$$
Since $V\not\equiv0$, $h$ is not constant, and thus $d\log h\not\equiv 0$. Thus, \eqref{Hardy} holds with a non-zero $\rho=|\nabla \log h|^p$, and it follows that $M$ is $p-$hyperbolic.\cqfd

\noindent \textit{Proof of Corollary \ref{main3}:} if \eqref{S} holds, then for $\mathcal{V}_-\in L^{\frac{n}{2}}$, strong subcriticality is equivalent to subcriticality of $\Delta+\mathcal{V}$ (see the Preliminaries). Thus, there exists $\varepsilon>0$ such that 

$$\varepsilon\int_M|\nabla u|^2\leq \int_M|\nabla u|^2+\mathcal{V}u^2,\qquad \forall u\in C_0^\infty(M).$$
Hence, $\Delta+\mathcal{V}$ satisfies the Sobolev inequality

$$||u||_{\frac{2n}{n-2}}\leq C\int_M|\nabla u|^2+\mathcal{V}u^2,\qquad \forall u\in C_0^\infty(M).$$
By Example \ref{Sob_Kato}, a potential $\mathcal{V}$ belonging to $L^{\frac{n}{2}\pm\varepsilon}$ is in $K^\infty(M)$, therefore by Theorem \ref{heat_kernel}, $e^{-t(\Delta+\mathcal{V})}$ has Gaussian estimates, and thus is bounded uniformly on $L^1$. According to \cite[Theorem 2.7]{CSCV}, the Sobolev inequality for $\Delta+\mathcal{V}$ and the fact that $e^{-t(\Delta+\mathcal{V})}$ is bounded uniformly on $L^1$ implies that for every $p\in (1,n)$, $(\Delta+\mathcal{V})^{-1/2}$ is bounded from $L^p$ to $L^q$, $\frac{1}{q}=\frac{1}{p}-\frac{1}{n}$. From the result of Theorem \ref{main1}, $d(\Delta+\mathcal{V})^{-1/2}$ is bounded on $L^p$ if and only if $(d h)(\Delta+\mathcal{V})^{-1/2}$ is bounded on $L^p$. Thus, we see that in order to conclude the proof, it is enough to show that $d h\in L^n$. This will be proved in Lemma \ref{h_prop}.

\cqfd

\noindent \textit{Proof of Theorem \ref{main4}:} by Corollary \ref{main2}, $d(\Delta+\mathcal{V})^{-1/2}$ is unbounded on $L^p$ if $M$ is $p$-parabolic. By Theorem \ref{capacity}, the hypothesis on capacities implies that for every $p<\kappa$, the following reverse volume estimate holds:

\begin{equation}\label{reverse}
r^p\lesssim \frac{V(x,r)}{V(x,1)},\qquad \forall x\in M,\,\forall r>0.
\end{equation}
Since $\mathcal{V}\in L^{\frac{\nu'}{2}-\varepsilon}\left(M,\frac{d\mu(x)}{V(x,1)}\right)\cap L^{\frac{\nu}{2}+\varepsilon}\left(M,\frac{d\mu(x)}{V(x,1)}\right)$, by interpolation and Lemma \ref{para_exp}, one obtains that $\mathcal{V}\in L^{\frac{p}{2}}\left(M,\frac{d\mu(x)}{V(x,1)}\right)$, for every $p<\kappa$, close enough to $\kappa$. Therefore, one sees that for $p<\kappa$, close enough to $\kappa$,

$$\int_1^\infty  \left|\left|\frac{|\mathcal{V}|^{1/2}}{V(\cdot,t)^{1/p}}\right|\right|_p\,dt<\infty.$$
Also, by Proposition \ref{cond_Kato}, $\mathcal{V}\in K^\infty(M)$. Thus, by Theorem \ref{AO}, the Riesz transform $d(\Delta+\mathcal{V})^{-1/2}$ is bounded on $L^p$ for all $1<p<\min(\kappa,p_0)$. This concludes the proof of Theorem \ref{main4}.

\cqfd
To prove Lemma \ref{dL}, we will need the following result concerning the positive function $h$, solution of $(\Delta+\mathcal{V})u=0$, whose existence is provided by Theorem \ref{pos_sol}.

\begin{Lem}\label{h_prop}

The function $h$ satisfies the following properties:

\begin{enumerate}

\item There exists $q\in [1,\infty)$ such that $h-1\in L^q\left(M,\frac{d\mu(x)}{V(x,1)}\right)$.

\item Under assumption (i),  $dh\in L^\infty$.

\item Under assumption (ii), $dh\in L_V^r$, for $r< p_0$, close enough to $p_0$.

\item Under the assumptions of Corollary \ref{main3}, $dh\in L^n$.

\end{enumerate}

\end{Lem}

\begin{proof}

By Theorem \ref{pos_sol}, $h$ satisfies the following equation:

$$1=h+\Delta^{-1}\mathcal{V}h.$$
Therefore, $|1-h|\leq\Delta^{-1}|\mathcal{V}|h\lesssim \Delta^{-1}|\mathcal{V}|$. By Proposition \ref{cond_Kato}, there exists $q\in [1,\infty)$ such that $\Delta^{-1}|\mathcal{V}|\in L^q_V$. Consequently, $1-h\in L^q_V$. Again according to Proposition \ref{cond_Kato}, $\mathcal{V}$ is in $K^\infty(M)$, and by definition of $K^\infty(M)$, $\Delta^{-1}|\mathcal{V}|\in L^\infty$. 

Let us now assume (i). By the gradient estimate of Cheng-Yau (see e.g. \cite[Theorem 6.1]{Li}), as a consequence of the bound from below of the Ricci curvature, there is a constant $C$ such that for every $x\in M$,

\begin{equation}\label{H1}
|\nabla \log G(x,y)|\leq C,\qquad \forall y\in M\setminus B(x,1),
\end{equation}
and,

\begin{equation}\label{H2}
|\nabla \log G(x,y)|\leq Cd(x,y)^{-1},\qquad\forall y\in B(x,1)\setminus\{x\}.
\end{equation}
As a consequence of \eqref{D} and \eqref{UE}, there holds:

$$G(x,y)\leq \int_{d(x,y)}^\infty \frac{rdr}{V(x,r)}$$
(see e.g. \cite[Exercise 15.8]{Grig}). Using \eqref{exp_D} and the fact that $\nu'>2$, we see that there is constant $C$ such that for all $x\in M$,

$$\int_1^\infty \frac{rdr}{V(x,r)}\leq \frac{C}{V(x,1)}.$$
Also, by Bishop-Gromov and our assumption on Ricci, for every $x\in M$ and $r\leq 1$,

\begin{equation}\label{BG}
\frac{V(x,r)}{V(x,1)}\geq Cr^N,\,\,V(x,r)\leq Cr^N,
\end{equation}
where $N$ is the dimension of $M$. Therefore, for all $(x,y)\in M^2$ such that $d(x,y)\leq 1$,

$$G(x,y)\leq C\frac{d(x,y)^{2-N}}{V(x,1)}+\frac{C}{V(x,1)}\leq C\frac{d(x,y)^{2-N}}{V(x,1)}.$$
This implies, using \eqref{H2} and \eqref{BG}, that for all $q<\frac{N}{N-1}$, there exists some constant $C_q$ so that for every $x\in M$,

\begin{equation}\label{grad_q}
||\nabla G(x,\cdot)||_{L^q(B(x,1))}\leq C_qV(x_0,1)^{-\frac{1}{q}}.
\end{equation}
Using the hypothesis that $\mathcal{V}$ satisfies \eqref{L_p} and \eqref{grad_q} for $q=p'<\frac{N}{N-1}$, we get that there is a constant $C$ such that for every $x\in M$,

\begin{equation}\label{dG_local}
\int_{B(x,1)}|\nabla G(x,y)|\,|\mathcal{V}(y)|dy\leq C.
\end{equation}
Also, by \eqref{H1} and the fact that $\mathcal{V}$ is in $K^\infty(M)$, 

\begin{equation}\label{dG_glob}
\sup_{x\in M}\int_{M\setminus B(x,1)}|\nabla G(x,y)|\,|\mathcal{V}(y)|dy<\infty.
\end{equation}
Combining \eqref{dG_local} and \eqref{dG_glob}, one obtains that $dh\leq d\Delta^{-1}\mathcal{V}h\in L^\infty$. 

Let us now assume (ii). We write

$$||dh||_{L_V^r}\leq \int_0^1 ||\nabla e^{-t\Delta}||_{L_V^{s_1},L_V^r}||\mathcal{V}||_{L_V^{s_1}}\,dt+\int_1^\infty ||\nabla e^{-t\Delta}||_{L_V^{s_2},L_V^r}||\mathcal{V}||_{L_V^{s_2}}\,dt.$$
For $r< p_0$, using Corollary \ref{grad_weight}, one finds

$$||dh||_{L_V^r}\leq ||\mathcal{V}||_{L_V^{s_1}}\int_0^1 \varphi_{s_1,r}(t)\,\frac{dt}{\sqrt{t}}+||\mathcal{V}||_{L_V^{s_2}}\int_1^\infty  \varphi_{s_2,r}(t)\,\frac{dt}{\sqrt{t}}.$$
By definition of $\varphi_{p,q}$, the two integrals converge if and only if

$$-\frac{\nu}{2s_1}+\frac{\nu'}{2r}>-\frac{1}{2},\,-\frac{\nu'}{2s_2}+\frac{\nu}{2r}<-\frac{1}{2}.$$
Since $r$ is arbitrarily close to $p_0$, it is enough to have these two inequalities satisfied for $r=p_0$, i.e.

$$-\frac{\nu}{s_1}+\frac{\nu'}{p_0}>-1,\,-\frac{\nu'}{s_2}+\frac{\nu}{p_0}<-1.$$
This is equivalent to

$$s_2<\frac{\nu' p_0}{p_0+\nu},\,s_1>\frac{\nu p_0}{p_0+\nu'}.$$
We now choose $s_2=\frac{\nu' p_0}{p_0+\nu}-\varepsilon$, and $s_1=\frac{\nu p_0}{p_0+\nu'}+\varepsilon$. By hypothesis, $\mathcal{V}\in L_V^{s_1}\cap L_V^{s_2}$, and the result is proved.

\end{proof}

\noindent \textit{Proof of Lemma \ref{dL}:} for $p=2$, by a direct consequence of the Green formula, $d\Delta_{h^2}^{-1/2}$ is an isometry on $L^2(M,h^2\,d\mu)\simeq L^2(M,\,d\mu)$. For $p>2$, we want to apply Proposition \ref{C-D}, with $A=h^2$. Indeed, assuming that the hypotheses of Proposition \ref{C-D} are satisfied for the choice $A=h^2$, we obtain that $d\mathcal{L}^{-1/2}=d\Delta_{h^2}^{-1/2}$ is bounded on $L^p$ for every $p\in (2,p_0)$, hence the result of Lemma \ref{dL}. It thus remains to check that the hypotheses of Proposition \ref{C-D} are fulfilled. We check this in the following sequence of lemmas.

%the following perturbation result of T. Coulhon and N. Dungey:

%\begin{Thm}\cite[Theorem 4.1]{CD}\label{CD}

%Assume that $\mu$ is a positive function such that $\mu\sim 1$ and $1-\mu\in L^q$ for some $q\in [1,\infty)$, that $d\Delta^{-1/2}$ and $d(\Delta_\mu+1)^{-1/2}$ are bounded on $L^p$ for $p\in (2,p_0)$. Assume also that for some $\alpha>0$, 

%$$||e^{-t\Delta}||_{1,\infty}\leq Ct^{-\alpha}\qquad \forall t>1,$$
%and that $d\Delta_\mu^{-1/2}$ is bounded on $L^p$ for every $p\in (p_0',2)$. Then $d\Delta_\mu^{-1/2}$ is bounded on $L^p$ for every $p\in (2,p_0)$.

%\end{Thm}

\begin{Lem}\label{1<p<2}
Under the assumptions of Theorem \ref{main1}, the Riesz transform associated to the weighted Laplacian $d\Delta_{h^2}^{-1/2}$ is bounded on $L^p$ for every $1<p\leq 2$.

\end{Lem}
\begin{proof}
The boundedness on $L^2(M,\,d\mu)\simeq L^2(M, \,h^2d\mu)$ follows from the fact that by the Green formula, $d \Delta_{h^2}^{-1/2}$ is an isometry on $L^2(\Omega,\, h^2d\mu)$. Since $h\sim 1$, the relative Faber-Krahn inequality \eqref{RFK} is satisfied for the weighted Laplacian $\Delta_{h^2}$, and consequently the heat kernel $e^{-t\Delta_{h^2}}$ has Gaussian upper-estimates. Also, the measure $h^2d\mu$ is doubling since $d\mu$ is. By \cite[Theorem 1.1] {CDu}, the Riesz transform $d \Delta_{h^2}^{-1/2}$ is bounded on $L^p$ for every $1<p\leq 2$.

\end{proof}

\begin{Lem}\label{local}

Under assumptions {\em (i)} or {\em (ii)}, the local Riesz transform $d(\Delta_{h^2}+1)^{-1/2}$ is bounded on $L^p$, for all $p\in (2,p_0)$.

\end{Lem}

\begin{proof}
Notice that $d(\Delta_{h^2}+1)^{-1/2}$ is bounded on $L^2$. By interpolation, it is thus enough to prove the boundedness of $d(\Delta_{h^2}+1)^{-1/2}$ on $L^p$ for $p$ close enough to $p_0$. Thus, for the rest of the proof, we assume that $p<p_0$ is close enough to $p_0$. Conjugating by $h$ and using $h\sim 1$, we easily see that the boundedness of $d(\Delta_{h^2}+1)^{-1/2}$ on $L^p$ is equivalent to 

$$D(\Delta+\mathcal{V}+1)^{-1/2} : L^p\rightarrow L^p,$$
where we recall that $D$ is the differential operator of degree one defined by

$$D=d\circ h^{-1}.$$
Also,

\begin{equation}\label{decomp}
D(\Delta+\mathcal{V}+1)^{-1/2}=d(h^{-1})(\Delta+\mathcal{V}+1)^{-1/2}+h^{-1}d(\Delta+\mathcal{V}+1)^{-1/2}.
\end{equation}
It is thus enough to prove that $(dh)(\Delta+\mathcal{V}+1)^{-1/2}$ and $d(\Delta+\mathcal{V}+1)^{-1/2}$ are bounded on $L^p$. We start with $d(\Delta+\mathcal{V}+1)^{-1/2}$. By a straightforward adaptation of the proof of \cite[Theorem 3.9]{AO}, if $\mathcal{V}$ satisfies

\begin{equation}\label{cond_AO}
\int_0^1\left|\left|\frac{|\mathcal{V}|^{1/2}}{V(\cdot,\sqrt{t})^{\frac{1}{r_1}}}\right|\right|_{r_1}\,\frac{dt}{\sqrt{t}}+\int_1^\infty\left|\left|\frac{|\mathcal{V}|^{1/2}}{V(\cdot,\sqrt{t})^{\frac{1}{r_2}}}\right|\right|_{r_2}e^{-t}\,\frac{dt}{\sqrt{t}}<\infty,
\end{equation}
then $(\Delta+1)^{1/2}(\Delta+\mathcal{V}+1)^{-1/2}$ is bounded on $L^r$, $r\in (1,\min(r_1,r_2))$. The additional $e^{-t}$ in the second integral comes from the fact that one considers $\Delta+1$ and $\Delta+\mathcal{V}+1$ instead of $\Delta$ and $\Delta+\mathcal{V}$ as in \cite{AO}. Using \eqref{exp_D} and the hypothesis on $\mathcal{V}$, one sees that \eqref{cond_AO} is satisfied for $r_1=r_2=\nu+2\varepsilon$ in case (i), and $r_1=r_2=p_0$ in case (ii). Also, by analyticity of $e^{-t\Delta}$ on $L^p$, $\Delta^{-1/2}(\Delta+1)^{-1/2}$ is bounded on $L^p$. Thus, writing

$$d(\Delta+\mathcal{V}+1)^{-1/2}=\left(d\Delta^{-1/2}\right)\left(\Delta^{1/2}(\Delta+1)^{-1/2}\right)\left((\Delta+1)^{1/2}(\Delta+\mathcal{V}+1)^{-1/2}\right),$$
we see that $d(\Delta+\mathcal{V}+1)^{-1/2}$ is bounded on $L^p$ for all $p\in (1,p_0)$. Let us now treat the operator $(dh)(\Delta+\mathcal{V}+1)^{-1/2}$, for which we distinguish two cases.\\

\noindent{\em Case 1: assume that }(i) {\em is satisfied:} let us first notice that $(\Delta+\mathcal{V}+1)^{-1/2}$ is bounded on $L^p$: indeed, 

\begin{equation}\label{I1}
||(\Delta+\mathcal{V}+1)^{-1/2}||_{p,p}\leq C\int_0^\infty e^{-t}||e^{-t(\Delta+\mathcal{V})}||_{p,p}dt,
\end{equation}
furthermore by Theorem \ref{heat_kernel}, $e^{-t(\Delta+\mathcal{V})}$ has Gaussian estimates and so it is uniformly bounded on $L^p$. So, the integral in \eqref{I1} converges, therefore $(\Delta+\mathcal{V}+1)^{-1/2}$ is bounded on $L^p$. By Lemma \ref{h_prop}, $d h\in L^\infty$. Therefore, $(dh)(\Delta+\mathcal{V}+1)^{-1/2}$ is bounded on $L^p$.\\

\noindent{\em Case 2: assume that }(ii) {\em is satisfied:} write

$$(dh)(\Delta+\mathcal{V}+1)^{-1/2}=\int_0^\infty (dh)e^{-t(\Delta+\mathcal{V})}e^{-t}\,\frac{dt}{\sqrt{t}}.$$
Since $e^{-t(\Delta+\mathcal{V})}$ has Gaussian estimates, by Corollary \ref{heat_weight} and H\"{o}lder's inequality, one has

$$||(dh)(\Delta+\mathcal{V}+1)^{-1/2}||_{p,p}\leq ||dh||_{L^r_V} \int_0^\infty \varphi_{p,q}(t)e^{-t}\,\frac{dt}{\sqrt{t}},$$
with $\frac{1}{r}=\frac{1}{p}-\frac{1}{q}$. The integral $\int_0^\infty \varphi_{p,q}(t)e^{-t}\,\frac{dt}{\sqrt{t}}$ converges if and only if 

\begin{equation}\label{CO}
\frac{\nu'}{2q}-\frac{\nu}{2p}>-\frac{1}{2}.
\end{equation}
Since $p_0>\nu$, \eqref{CO} is satisfied if $p$ is close enough to $p_0$, and $q$ is big enough. By Lemma \ref{h_prop}, $dh\in L_V^r$, for $r< p_0$ close enough to $p_0$. Taking $p$ close enough to $p_0$, and $q$ big enough, one can arrange for $r<p_0$ as close as we want to $p_0$, as well as \eqref{CO} satisfied. Therefore, the operator $(dh)(\Delta+\mathcal{V}+1)^{-1/2}$ is bounded on $L^p$.

%: denote $A_0=\Delta+1$ and $A=\Delta+\mathcal{V}+1$, then

%\begin{equation}\label{I2}
%\begin{array}{rcl}
%I-A_0^{1/2}A^{-1/2}&=&\int_0^\infty A_0^{1/2}\left((I+tA_0)^{-1}-(I+tA)^{-1}\right)t^{-1/2}dt\\\\
%&=&\int_0^\infty (tA_0)^{1/2}(I+tA_0)^{-1}V(I+tA)^{-1}dt\\\\
%\end{array}
%\end{equation}
%Writing $(A+\lambda)^{-\alpha}=\int_0^\infty e^{-t(A+\lambda)}t^{\alpha-1}$, and using that $||e^{-tA}||_{r,r}\leq C e^{-t}$, one gets

%$$||(I+tA)^{-1}||_{r,r}\leq \frac{C}{1+t}$$
%and

%$$||(I+tA_0)^{-1/2}||_{r,r}\leq \frac{C}{(1+t)^{1/2}}.$$
%Furthermore, by analyticity of the heat semi-group,

%$$||(tA_0)^{1/2}(I+tA_0)^{-1/2}||_{r,r}\leq C.$$
%Also, since $V\in L^\infty$ the operator ``multiplication by $V$'' is bounded on $L^r$, and therefore by composition

%$$||(tA_0)^{1/2}(I+tA_0)^{-1}V(I+tA)^{-1}||_{r,r}\leq \frac{C}{(1+t)^{3/2}},$$
%hence the integral \eqref{I2} converges in $||\cdot||_{r,r}$ norm, and we get that $(\Delta+1)^{1/2}(\Delta+\mathcal{V}+1)^{-1/2}$ is bounded on $L^r$, for every $r\in (1,\infty)$. Since by hypothesis the Riesz transform is bounded on $L^{p+\varepsilon}$, we conclude that $d(\Delta+\mathcal{V}+1)^{-1/2}$ is bounded on $L^r$ for every $r\in (2,p+\varepsilon)$.

\end{proof}

\begin{center}{\bf Acknowledgments} \end{center}
It is a pleasure to thank Y. Pinchover for interesting discussions, and G. Carron for comments on the manuscript. While this work was being done, the author was supported by post-doctoral fellowships at the Technion (Haifa, Israel) and at the University of British Columbia (Vancouver, Canada). He thanks both institution for having provided him with excellent research environment.


\begin{thebibliography}{99}

\bibitem{A} J.~Assaad, Riesz transforms associated to Schr\"{o}dinger operators with negative potentials, {\em Publ. Mat.} {\bf 55} (2011), no. 1, 123--150.

\bibitem{AO} J.~Assaad, E.M.~Ouhabaz, Riesz transform of Schr\"{o}dinger operators on manifolds, {\em Journal of Geometric Analysis} {\bf 22} (2012), no. 4, 1108--1136.

\bibitem{AB} P.~Auscher, B.~Ben-Ali, Maximal inequalities and Riesz transform estimates on $L^p$ spaces for Schr\"{o}dinger operators with nonnegative potentials, {\em Ann. Inst. Fourier (Grenoble)} {\bf 57} (2007), no. 6, 1975--2013. 

\bibitem{ACDH} P.~Auscher, T.~Coulhon, X.T.~Duong, S.~Hofmann, Riesz transform on manifolds and heat kernel regularity, {\em Ann. Sci. \'{E}cole Norm. Sup.} (4) {\bf 37} (2004), no. 6, 911--957. 

\bibitem{Bak} D.~Bakry, \'{E}tude des transformations de Riesz dans les vari\'{e}t\'{e}s Riemanniennes \`{a} courbure de Ricci minor\'{e}e, S\'{e}minaire de Probabilit\'{e}s, XXI, 137--172, {\em Lecture Notes in Math.}, {\bf 1247}, Springer, Berlin, 1987.

\bibitem {BBA} N.~Badr, B.~Ben-Ali, $L^p$ boundedness of the Riesz transform related to Schr\"{o}dinger operators on a manifold, {\em Ann. Sc. Norm. Super. Pisa Cl. Sci.} (5) {\bf 8} (2009), no. 4, 725--765.



\bibitem{B} P.~Buser, A note on the isoperimetric constant, {\em Ann. Sci. École Norm. Sup.} (4) {\bf 15} (1982), no. 2, 213--230.

\bibitem{BCS} S.~Boutayeb, T.~Coulhon, A.~Sikora, A new approach to pointwise heat kernel upper bounds on doubling metric measure spaces, {\em Adv. in Math.}, {\bf 270} (2015) 302--374.

\bibitem{Ch} Z.-Q.~Chen, Gaugeability and conditional gaugeability, {\em Trans. Amer. Math. Soc.} {\bf 354} (2002), no. 11, 4639-–4679.

%\bibitem{Ca2} G.~Carron, In\'{e}galit\'{e}s isop\'{e}rim\'{e}triques de Faber-Krahn et cons\'{e}quences, {\em Actes de la Table Ronde de G\'{e}om\'{e}trie Diff\'{e}rentielle (Luminy, 1992)}, 205--232, S\'{e}min. Congr., 1, Soc. Math. France, Paris, 1996.

%\bibitem{Ca3} G.~Carron, In\'{e}galit\'{e}s de Faber-Krahn et inclusion de Sobolev-Orlicz, {\em Potential Anal.} {\bf 7} (1997), no. 2, 555--575.

\bibitem{Ca} G.~Carron, Riesz transform on manifolds with quadratic curvature decay, arXiv:1403.6278.

\bibitem{CCH} G.~Carron, T.~Coulhon, A.~Hassell, Riesz transform and $L^p$-cohomology for manifolds with Euclidean ends, {\em Duke Math. J.} {\bf 133} (2006), no. 1, 59--93.

\bibitem{C} T.~Coulhon, Dimensions at infinity for Riemannian manifolds, {\em Potential Anal.} {\bf 4} (1995), no. 4, 335–-344.

\bibitem{CDS} T.~Coulhon, B.~Devyver, A.~Sikora, in preparation.

\bibitem{CDu} T.~Coulhon, X.T..~Duong, Riesz transforms for $1\leq p\leq 2$, {\em Trans. Amer. Math. Soc.} {\bf 351} (1999), no. 3, 1151-–1169.

\bibitem{CDu2} T.~Coulhon, X.T..~Duong, Riesz transform and related inequalities on noncompact Riemannian manifolds, {\em Comm. Pure Appl. Math.} {\bf 56} (2003), no. 12, 1728–-1751.

\bibitem{CD} T.~Coulhon, N.~Dungey, Riesz transform and perturbation, {\em J. Geom. Anal.} {\bf 17} (2007), no. 2, 213–-226. 

\bibitem{CGT} J.~Cheeger, M.~Gromov, M.~Taylor, Finite propagation speed, kernel estimates for functions of the Laplace operator, and the geometry of complete Riemannian manifolds, {\em J. Differential Geom.} {\bf 17} (1982), no. 1, 15--53.

\bibitem{CHSC} T.~Coulhon, I.~Holopainen, L.~Saloff-Coste, Harnack inequality and hyperbolicity for subelliptic $p$-Laplacians with applications to Picard type theorems, {\em Geom. Funct. Anal.} {\bf 11} (2001), no. 6, 1139–-1191.

\bibitem{CMO} P.~Chen, J.~Magniez, E.-M.~Ouhabaz,Riesz transforms on non-compact manifolds, arXiv:1411.0137.

\bibitem{CSC} T.~Coulhon, L.~Saloff-Coste, Vari\'{e}t\'{e}s Riemanniennes isom\'{e}triques \`{a} l'infini, {\em Rev. Mat. Iberoamericana} {\bf 11} (1995), no. 3, 687–-726.

\bibitem{CSC2} T.~Coulhon, L.~Saloff-Coste, Isop\'{e}rim\'{e}trie pour les groupes et les vari\'{e}t\'{e}s, {\em Rev. Mat. Iberoamericana} {\bf 9} (1993), no. 2, 293--314.

\bibitem{CSCV} T.~Coulhon, L.~Saloff-Coste, N.~Varopoulos, Analysis and geometry on groups, {\em Cambridge Tracts in Mathematics}, {\bf 100}, Cambridge University Press, Cambridge, 1992.

\bibitem{CZ} T.~Coulhon, Q.S.~Zhang, Large time behavior of heat kernels on forms, {\em J. Differential Geom.} {\bf 77} (2007), no. 3, 353–-384.

\bibitem{DS} E.B.~Davies, B.~Simon, $L^p$ norms of noncritical Schr\"{o}dinger semigroups, {\em J. Funct. Anal.} {\bf 102} (1991), no. 1, 95–-115.

\bibitem{D0} B.~Devyver, On the finiteness of the Morse index for Schr\"{o}dinger operators, {\em Manuscripta Math.} {\bf 139} (2012), no. 1-2, 249–-271.

\bibitem{D1} B.~Devyver, A Gaussian estimate for the heat kernel on differential forms and application to the Riesz transform, {\em Math. Ann.} {\bf 358} (2014), no. 1-2, 25–-68.

\bibitem{D2} B.~Devyver, A perturbation result for the Riesz transform, to appear in {\em Ann. Sc. Norm. Super. Pisa Cl. Sci.}, DOI 10.2422/2036-2145.201107\textunderscore \,010.

\bibitem{DFP} B.~Devyver, M.~Fraas, Y.~Pinchover, Optimal Hardy weight for second-order elliptic operator: an answer to a problem of Agmon, {\em J. Funct. Anal.} {\bf 266} (2014), no. 7, 4422–-4489.

\bibitem{G2} A.~Grigor'yan, Stochastically complete manifolds, {\em Dokl. Akad. Nauk SSSR} {\bf 290} (1986), no. 3, 534–-537.

\bibitem{G} A.~Grigor'yan, Heat kernels on weighted manifolds and applications, The ubiquitous heat kernel, 93–191, {\em Contemp. Math.}, {\bf 398}, Amer. Math. Soc., Providence, RI, 2006.

\bibitem{G3} A.~Grigor'yan, Analytic and geometric background of recurrence and non-explosion of the Brownian motion on Riemannian manifolds, {\em Bull. Amer. Math. Soc.} {\bf 36} (1999), no. 2, 135--249.

\bibitem{Grig} A.~Grigor'yan, Heat Kernel and Analysis on Manifolds, {\em AMS/IP Studies in Advanced Mathematic}s, {\bf 47}, American Mathematical Society, Providence, RI; International Press, Boston, MA, 2009.

\bibitem{G4} A.~Grigor'yan, Heat kernel upper bounds on a complete non-compact manifold, {\em Rev. Mat. Iberoamericana} {\bf 10} (1994), no. 2, 395-–452.

\bibitem{GH} A.~Grigor'yan, J. Hu, Upper bounds of heat kernels on doubling spaces, http://www.math.uni-bielefeld.de/grigor/pubs.htm.

\bibitem{GuH} C.~Guillarmou, A.~Hassell, Resolvent at low energy and Riesz transform for Schr\"{o}dinger operators on asymptotically conic manifolds I, {\em Math. Ann.} {\bf 341} (2008), no. 4, 859--896.

\bibitem{H} I.~Holopainen, Volume growth, Green's functions, and parabolicity of ends, {\em Duke Math. J.} {\bf 97} (1999), no. 2, 319-–346.

\bibitem{Leb} C.~LeBrun,  Complete Ricci-flat K\"{a}hler metrics on $C^n$ need not be flat, Several complex variables and complex geometry, Part 2 (Santa Cruz, CA, 1989), 297--304, {\em Proc. Sympos. Pure Math.}, {\bf 52}, Part 2, Amer. Math. Soc., Providence, RI, 1991.

\bibitem{Li} P.~Li, {\em Geometric analysis}, Cambridge Studies in Advanced Mathematics, {\bf 134}, Cambridge University Press, Cambridge, 2012.

\bibitem{P1}  Y.~Pinchover, On positive solutions of second-order elliptic equations, stability results, and classification, {\em Duke Math. J.} {\bf 57} (1988), no. 3, 955-980.

\bibitem{P2} Y.~Pinchover, Criticality and ground states for second-order elliptic equations, {\em J. Differential Equations} {\bf 80} (1989), no. 2, 237–250. 

\bibitem{P5} Y.~Pinchover, On criticality and ground states of second order elliptic equations II, {\em J. Differential Equations} {\bf 87} (1990), no. 2, 353-–364.

\bibitem{P4} Y.~Pinchover, On the equivalence of Green functions of second order elliptic equations in $\R^n$, {\em Differential Integral Equations} {\bf 5} (1992), no. 3, 481-–493.

\bibitem{P3} Y.~Pinchover, Maximum and anti-maximum principles and eigenfunctions estimates via perturbation theory of positive solutions of elliptic equations, {\em Math. Ann.} {\bf 314} (1999), no. 3, 555–-590.

\bibitem{P6}  Y.~Pinchover, Topics in the theory of positive solutions of second-order elliptic and parabolic partial differential equations, in "Spectral Theory and Mathematical Physics: A Festschrift in Honor of Barry Simon's 60th Birthday",  eds. F. Gesztesy, et al., {\em Proceedings of Symposia in Pure Mathematics} {\bf 76}, American Mathematical Society, Providence, RI, 2007, 329--356.

\bibitem{PT1} Y.~Pinchover, K.~Tintarev, A ground state alternative for singular Schr\"{o}dinger operators, {\em J. Funct. Anal.} {\bf 230} (2006), no. 1, 65--77.

\bibitem{PT2} Y.~Pinchover, K.~Tintarev, Ground state alternative for $p$-Laplacian with potential term, {\em Calc. Var. Partial Differential Equations} {\bf 28} (2007), no. 2, 179--201.

\bibitem{M} M.~Murata, Semismall perturbations in the Martin theory for elliptic equations, {\em Israel J. Math.} {\bf 102} (1997), 29--60.

 

\bibitem{SC} L.~Saloff-Coste, Aspects of Sobolev-type inequalities, {\em London Mathematical Society Lecture Note Series}, {\bf 289}, Cambridge University Press, Cambridge, 2002.

\bibitem{Sh} Z.W.~Shen, $L^p$ estimates for Schr\"{o}dinger operators with certain potentials, {\em Ann. Inst. Fourier (Grenoble)} {\bf 45} (1995), no. 2, 513--546. 

\bibitem{S} B.~Simon, Brownian motion, $L^p$ properties of Schrödinger operators and the localization of binding, {\em J. Funct. Anal.} {\bf 35} (1980), no. 2, 215–-229.

\bibitem{Sik} A.~Sikora, Riesz transform, Gaussian bounds and the method of wave equation, {\em Math. Z.} {\bf 247} (2004), no. 3, 643–-662.

\bibitem{T1} M.~Takeda, Gaussian bounds of heat kernels for Schrödinger operators on Riemannian manifolds, {\em Bull. Lond. Math. Soc.} {\bf 39} (2007), no. 1, 85–-94.

\bibitem{T2} M.~Takeda, Conditional gaugeability and subcriticality of generalized Schr\"{o}dinger operators, {\em J. Funct. Anal.} {\bf 191} (2002), no. 2, 343--376.

\bibitem{Z1} Q.S.~Zhang, On a parabolic equation with a singular lower order term II. The Gaussian bounds, {\em Indiana Univ. Math. J.} {\bf 46} (1997), no. 3, 989--1020. 

\bibitem{Z2} Q.S.~Zhang, Global bounds of Schr\"{o}dinger heat kernels with negative potentials, {\em J. Funct. Anal.} {\bf 182} (2001), no. 2, 344--370. 

\bibitem{Zhao} Z.~Zhao, Subcriticality and gaugeability of the Schrödinger operator,
{\em Trans. Amer. Math. Soc.} {\bf 334} (1992), no. 1, 75--96. 




%
\end{thebibliography}
\end{document}